%% file: Notes.tex
\documentclass[a4paper,
  headings=standardclasses
]{scrartcl}
\usepackage[a4paper,margin=2cm]{geometry}

\usepackage{mathtools}
\usepackage{amssymb}
\usepackage{amsthm}

\usepackage{algorithm}
\usepackage[noend]{algpseudocode}

\usepackage{tikz}
\usepackage{tikz-cd}

\usepackage{cleveref}

\usepackage{xcolor}

\usepackage[status=draft,layout=footnote]{fixme}
\fxsetup{theme=color}

\newcommand{\its}{{thin-sum}}

\newcommand{\ovs}[2]{{\overset{#1}{#2}}}

\newcommand{\NN}{\mathbb{N}}

\newcommand{\ZZ}{\mathbb{Z}}

\newcommand{\RR}{\mathbb{R}}

\newcommand{\cB}{\mathcal{B}}

\newcommand{\cC}{\mathcal{C}}
\newcommand{\cS}{\mathcal{S}}
\newcommand{\cW}{\mathcal{W}}

\newcommand{\cJ}{\mathcal{J}}

\newcommand{\tf}{\tilde f}
\newcommand{\fe}{\mathfrak{e}}

\DeclareMathOperator{\im}{im}

\DeclareMathOperator{\supp}{supp}

\newtheorem{theorem}{Theorem}[section]
\newtheorem{corollary}[theorem]{Corollary}
\newtheorem{lemma}[theorem]{Lemma}
\newtheorem{proposition}[theorem]{Proposition}
\theoremstyle{definition}

\newtheorem{remark}[theorem]{Remark}
\theoremstyle{definition}
\newtheorem{definition}[theorem]{Definition}

\newcommand{\point}[1]{{\par{\textit{#1}}}}

\newcommand{\Huf}{H^{\mathrm{uf}}}
\newcommand{\Hsuf}{H^{\mathrm{suf}}}
\newcommand{\Buf}{B^{\mathrm{uf}}}
\newcommand{\Bsuf}{B^{\mathrm{suf}}}
\newcommand{\Zuf}{Z^{\mathrm{uf}}}
\newcommand{\Zsuf}{Z^{\mathrm{suf}}}
\newcommand{\Cuf}{C^{\mathrm{uf}}}
\newcommand{\Csuf}{C^{\mathrm{suf}}}

\begin{document}

    \title{The first uniformly finite homology group with coefficients in ${\ZZ}$ and a characterisation of its vanishing in the transitive case}

    \author{R\'emi Bottinelli\footnote{Supported by the Swiss National Science Foundation project no. PP00P2-144681/1},  Tom Kaiser}

    \maketitle
    
    \begin{abstract}
        We study the first uniformly finite homology group of Block and Weinberger for uniformly locally finite graphs, with coefficients in $\ZZ$ and $\ZZ_2$.
        When the graph is a tree, or coefficients are in $\ZZ_2$, a characterisation of the group is obtained.
        In the general case, we describe three phenomena that entail non-vanishing of the group; their disjunction is shown to also be necessary for non-vanishing in the case of transitive graphs.
    \end{abstract}

    \tableofcontents

    \section{Introduction}   
    Uniformly finite homology, introduced in \cite{BW}, is a coarse invariant of well-behaved metric spaces.
    In~\cite{BW}, it is shown that vanishing of zeroth uniformly finite homology with coefficients in ${\ZZ}$ or ${\RR}$ is equivalent to expansion (non-amenability) of the space at hand.
    Higher uniformly finite homology groups don't enjoy, as far as we know, such clear cut descriptions.
    Partial results have been obtained, e.g.\ in~\cite{BD, DN}.
   
    In these notes, we focus on the first homology group, with coefficients in ${\ZZ}$ of (tame enough) graphs: $\Huf_1({\cdot},{\ZZ})$.
    Informally, three relatively independent phenomena are responsible for the appearance of homology classes in dimension one:
    \begin{enumerate}
        \item 
            Graph theoretical ends.
        \item
            The existence of circuits of arbitrarily large girth.
            More precisely, the non-existence of a uniform constant $R$ such that any circuit can be written as a (possibly infinite) sum of circuits of length ${\leq}R$. 
        \item
            A lack of $2$-simplices in the clique complex.
            This should be understood quantitatively, and likened to amenability.
    \end{enumerate}

    After recalling the necessary definitions in~\Cref{section:preliminaries}, we proceed in four parts.
   
    In~\Cref{section:trees}, we study the first of the above phenomena.
    Starting with a classification of the (first) homology of trees (\Cref{theorembips}) in terms of their ends, we then show that the homology of an arbitrary graph is bounded from below by that of an embedded tree that “realises the ends” (\Cref{tree_injection}).

    In~\Cref{section:Z2}, motivated by the observation that there is a surjection:
    \[
        \Huf_1({\cdot},{\ZZ}) \rightarrow \Huf_1({\cdot},{\ZZ}_2),
    \]
    we study $\Huf_1({\cdot},{\ZZ}_2)$ (note that the “uniformly finite” condition becomes vacuous in ${\ZZ}_2$ — this may sound like a bad omen, but turns out to clarify a lot of constructions).
    In particular, we get, in~\Cref{decomp_Z2},  a full description of $\Huf_1({\cdot},{\ZZ}_2)$ as a direct sum of two factors corresponding to exactly the first and second phenomena, that is, ends and “large circuits” respectively.
    Moreover, it is shown in~\Cref{thm_large_circuits_imply_inf_dim} that the existence of “large circuits” actually implies infinite-dimensionality of $\Huf_1({\cdot},{\ZZ}_2)$. 

    In~\Cref{section:ZZ_large_circuits}, again in $\ZZ$, we make precise the claim that “large circuits” imply non-zero homology.
    Since “large circuits” is a notion which depends on the coefficient rings, this result is \emph{not} a consequence of the previous section.

    In~\Cref{section:expansion}, inspired by a construction of~\cite{BS}, we define a notion of “expansion” in higher dimension, which aims at capturing the third phenomenon.
    It is easily verified (\Cref{expansion_implies_Huf_vanishes}) that our “expansion” implies vanishing of homology.
    As a (very partial) converse, we show (\Cref{transitive_graphs_expansion}) that in dimension one, and for vertex transitive graphs, vanishing implies “expansion”.

    Finally, refining the notion of expansion, we reach, for (infinite) transitive graphs (or more generally, infinite graphs with a cocompact action by their automorphism group), a characterisation of the non-vanishing of homology in terms of the three phenomena of (the refined) non-expansion, ends, and large circuits.

	\input{Preliminaries}

	\input{Trees_and_End_Defining_Trees}

	\input{Coefficients_in_Z2}
    \input{ZZ_Large_Circuits.tex}

	\input{Expansion}

    \bibliographystyle{unsrt} 
    \bibliography{biblio}
\end{document}

%% file: Preliminaries.tex

\section{Preliminaries}\label{section:preliminaries}

\subsection{Introduction}

    We follow the conventions of~\cite{Mosher}.

    If $X$ is a simplicial complex, it is:
    \begin{description}
        \item[Uniformly Locally Finite] (ULF) if there is a uniform bound on the cardinalities of the links of its vertices.
        \item[Uniformly Contractible] (UC) if $\forall r>0\ \exists  s(r)>0$ such that any set $A\subseteq X$ of diameter $\leq r$ is contractible in its $s(r)$-neighbourhood.
    \end{description}

\subsection{Definition of $\Hsuf$}

    If $X$ is a ULF simplicial complex, and $A$ is either $\RR $, $\ZZ$ or $\ZZ_2$, define:
    \[
        \Csuf_n(X,A) = l^\infty (X_{(n)},A),
    \]
    where $X_{(n)}$ is the set of non-degenerate $n+1$-simplices of $X$.
    and the boundary map:
    \begin{alignat*}{10}
        \partial : \Csuf_{n+1}(X) &\rightarrow  \Csuf_{n}(X)\\
        f                 &\mapsto  [\partial f: \tau  \mapsto  \sum_{\tau <\sigma } i(\tau :\sigma )f(\sigma )],
    \end{alignat*}
    where $i(\cdot :\cdot )$ takes care of alternating signs, after choosing an orientation for $X$.

    Then, the \emph{simplicial uniformly finite homology} of $X$ is defined as
    \[
        \Hsuf_*(X) = H(\Csuf_*(X)).
    \]

\subsection{Definition of $\Huf$}

    If $X$ is a ULF graph and $r\in \NN$, define the \emph{Rips Complex} of radius $r$ as having $n$-simplices:
    \[
        (R_r(X))_{(n)} = \{(x_1,…x_n) \in  X_{(0)}^{n} \ :\ d(x_i,x_j) \leq  r\},
    \]
    where $d$ is the graph distance (i.e. length of a shortest path).
    A simplex of $R_r(x)$ is seen as a “virtual” simplex in $X$.
    Then $R_1(X) = X$, $R_r{X} \subseteq  R_{r'}(X)$ whenever $r\leq r'$, and $R_r(X)$ is itself a ULF simplicial complex.
    It follows that we have a directed system:
    \[
        \left(\Csuf_{\star }(R_r(X))\right)_{r\in \NN}
    \]
    and can define:
    \[
        \Cuf_{\star }(X) := \lim_r \Csuf_{\star }(R_r(X))
    \]
    and finally, the \emph{uniformly finite homology} of $X$ is defined as:
    \[
        \Huf_{\star }(X) = H(\Cuf_{\star }(X))
    \]
    
    As a remark, note that $\Huf_{*}(X)$ can also be defined as the limit of the system 
    \[
        (\Hsuf_*(R_r(X)))_r.
    \]

    In the following, we will use 
    \[
        \Zuf,\Buf,\Zsuf,\Bsuf,
    \]
    for the kernels and image of the boundary maps in the chain complexes $\Cuf$ and $\Csuf$ respectively.

    The following two important facts are proved in~\cite{Mosher}.

    \begin{proposition}[$\Huf$ is QI-invariant.]
        $\Huf$ is a quasi-isometry invariant. 
    \end{proposition}
    \begin{proof}
        See~\cite[Step 3 in the proof of Theorem 12]{Mosher}.
    \end{proof}
    
    \begin{proposition}
        If $X$ is ULF, UC, then $\Hsuf(X) \cong  \Huf(X)$.
    \end{proposition}
    \begin{proof}
        See~\cite[Step 2 in the proof of Theorem 12]{Mosher}.
    \end{proof}

\subsection{“Hands-on” definition of $\Huf_1$}

    The definition of uniformly finite homology is rather abstract; the goal in this section is to find more concrete description of $\Huf_1$.

    Let $X$ be a ULF graph, and $A$ either $\ZZ $ or ${\ZZ}_2$.

    An element of $\Zsuf_1(X,A)$ can be seen as a (uniformly bounded) \emph{flow} on $X$; that is, each (directed) edge takes a certain value, which we view as flowing through the edge.
    The condition of being a homological cycle states that at any vertex, the sum of flows of edges directed to the vertex is the same as the flow of edges directed out of the vertex (“in=out”). 
    The two building blocks for such flows are circuits and bi-infinite paths (which we will call “bips”):
    Indeed, any oriented, graph theoretical circuit in $X$ defines an element of $\Zsuf_1(X,A)$ by taking the sum of its (oriented) edges.
    Similarly, a bi-infinite path (that is, morphism of graphs from $\ZZ$ to $X$, injective on edges for convenience) also defines an element of $\Zsuf_1(X,A)$.
    Now, for $A$ either $\ZZ $ or ${\ZZ}_2$, any element of $f \in  \Zsuf_1(X,A)$ can be decomposed into a sum of circuits and bips by “following paths”, in such a way that there is a uniform bound on the number of circuits/bips passing over any given edge (Take a maximal multiset of circuits “below” f and subtract it; Take a maximal multiset of bips “below” the result and subtract it; Nothing remains). 
    From now on, we will not distinguish circuits and bips as paths in $X$ from their representatives in homology.

    If a (possibly infinite) family $(c_i)_{i\in I}$ of circuits is given, along with coefficients $({\lambda}_i\in A)_{i\in I}$, then we say they define a \emph{thin} sum if the family of circuits is locally finite (see \cite{BG} for a motivation of “thinness”).

    We can then define the following objects:
    \begin{alignat*}{10}
        \cC_r(X,A) &:= \text{``sums of circuits of length $\leq r$ with uniformly bounded coefficients in $A$''}\\
        \cC_\infty (X,A) &:= \cup_r \cC_r(X,A)\\
        \cC(X,A) &:= \text{``thin sums of circuits with coefficients in $A$ such that the result is uniformly bounded''}
    \end{alignat*}

    It is easy to see that all of those objects are actually vector subspaces of $\Zsuf_1(X,A)$.
    Note that in the case $A=\ZZ_2$, $\cC_r$ can be described simply as the space of sums of circuits of length $\leq r$, and $\cC$ as the space of thin sums of circuits.
    \newcommand{\fC}{\mathfrak{C}}
    Note finally that, letting $\fC_r$ stand for the set of circuits of length $\leq r$, and
    \[
        \fe :l^\infty (\fC_r) \rightarrow  l^\infty (EX) = \Csuf_1(X),
    \]
    extended linearly from $c \mapsto  \text{“its edges”}$, we have $\cC_r = \fe \left(l^\infty (\fC_r)\right)$, and $\fe $ is continuous w.r.t. the pointwise topologies.

    We start with a description of $\cC_\infty $ in terms of our chain complexes.

    \begin{lemma}\label{lemma:CZB}
        If $X$ is a ULF graph, then, for any $r>0$:
        \begin{alignat*}{10}
            \cC_r(X,A) &\subseteq  \Zsuf_1(X,A) \cap  \Bsuf_1(R_{r/2+1} X,A)\\
            \Zsuf_1(X,A) \cap  \Bsuf_1(R_r X,A) &\subseteq  \cC_{3r}(X,A).
        \end{alignat*}
        Consequently:
        \[
            \cC_\infty (X,A) = \Zsuf_1(X,A) \cap  \Buf_1(X,A).
        \]
    \end{lemma}
    \begin{proof}
        Drop the $A$s.

        If $f \in  \cC_r(X)$, $f = \fe \phi $ for some $\phi :\fC_r \rightarrow  A$ uniformly bounded.
        Obviously, $f\in \Zsuf_1(X)$.
        For any $c\in  \supp \phi $, one can “triangulate” $c$: that is, if $c$ is the circuit$(v_1,…,v_s)$ with $s\leq r$, we can consider the sum 
        \[
            \Delta c := \sum_{i=1}^s (v_1,v_i,v_{i+1}) \in  \Csuf_2(R_{2r+1}X),
        \]
        with addition $\mod s$.
        One sees that $\partial (\Delta c) = \fe c$, and letting
        \[
            \Delta \phi  := \sum_{c} g(c)\cdot (\Delta c),
        \]
        we get $\Delta \phi  \in  \Csuf_2(R_{r/2+1}X)$ by ULF, and $\partial (\Delta \phi ) = \fe \phi  = f$.
        Thus, $f\in \Bsuf_2(R_{r/2+1}X)$.

        Conversely, assume $f \in  \Zsuf_1(X) \cap  \Bsuf_1(R_r X)$.
        Let $g\in \Csuf_2(R_r X)$ such that $\partial g = f$.

        For any two vertices $u,v$ at distance $\leq r$, fix a shortest path $p_{u,v} = (u=u_0,u_1,…u_s=v)$ from $u$ to $v$.
        If $t := (u,v,w)$ is a triangle in $R_r X$, consider the circuit $Ot$ obtained by concatenating the paths $p_{u,v},p_{v,w},p_{w,u}$.
        By construction, $Ot \in  \fC_{3r}$.
        Let also 
        \[
            Og := \sum_{t \in  (R_r X)_{(3)}} g(t)\cdot (Ot)
        \]
        Then $Og \in  l^\infty (\fC_{3r})$, again by ULF, and by our assumption that $\partial g \in  \Zsuf_1(X)$, we get:
        \[
            \fe (Og) = \partial g.
        \]
        It follows that $f =\fe (Og) \in  \cC_{3r}$.

        The last equality of the lemma follows simply by taking the directed unions.

    \end{proof}

    We can now describe $\Huf_1(X)$ in a way that only deals with circuits and paths, and does not deal with Rips complexes:
    \begin{proposition}\label{H1uf_rewritten}
        If $X$ is a ULF graph, then, for any $r>0$, the composite homomorphism:
        \[
            {\Phi}_r : \Zsuf_1(X,A) \hookrightarrow  \Zsuf_1(R_r X,A) \twoheadrightarrow   \Hsuf_1(R_r X,A)
        \]
        is surjective and its kernel satisfies:
        \[
            \cC_{2r} \leq  \ker {\Phi}_r \leq  \cC_{3r}.
        \]
        
        Consequently, the composite homomorphism
        \[
            {\Phi}_\infty  : \Zsuf_1(X,A) \hookrightarrow  \Zuf_1(X,A) \twoheadrightarrow   \Huf_1(X,A)
        \]
        is surjective and has kernel 
        \[
            \ker {\Phi}_\infty  = \cC_\infty (X,A).
        \]
        In particular:
        \[
            \Huf_1(X,A) \cong  \frac{\Zsuf_1(X,A)}{\cC_\infty (X,A)}.
        \]
    \end{proposition}
    \begin{proof}
        \renewcommand{\tf}{\tilde f}
        Drop the $A$s.
        
        Fix $r>0$.
        We first verify surjectivity of ${\Phi}_r$.
        Fix any $f\in \Zsuf_1(R_rX)$.
        Choose for any two $u,v$ in $X$, a unique shortest path $p_{u,v} = (u=u_0,u_1,…,u_s=v)$ from $u$ to $v$, and let:
        \[
            \tf := \sum_{(u,v) \in  (R_rX)_{(2)}} f(u,v) \cdot  p_{u,v},
        \]
        that is, $\tf$ is obtained by replacing “virtual edges” (i.e. edges in $(R_rX)_{(2)}$) by paths.
        By ULF, $\tf \in  \Csuf_1(X)$ and $\partial \tf = \partial f$, so that $\partial \tf=0$, hence $\tf \in  \Zsuf_1(X)$.
        We claim that $\tf - f \in  \Bsuf_1(R_rX)$, which will clearly imply surjectivity of ${\Phi}_r$.
        Indeed, if for any $(u,v)$ at distance $\leq r$, one considers:
        \[
            \Delta (u,v) := \sum_{i=1}^{s-1} (u_0,u_i,u_{i+1}) \in  \Csuf_2(R_rX) 
        \]
        where $(u=u_0,u_1,…,u_s=v)$ is a path from $u$ to $v$, then $\partial \Delta (u,v) = (u,v) - p_{u,v}$.
        Summing over the virtual edges $(u,v)$ in the support of $f$, we get:
        \[
            \partial \left( \sum_{(u,v) \in  (R_rX)_{(1)}} f(u,v)\cdot \Delta (u,v)  \right) = f - \tf.
        \]
        This shows that, indeed, ${\Phi}_r$ is surjective, assuming the sum in the parentheses is in $\Csuf_2(R_rX)$ which holds, as usual, by ULF.
        Now, the inclusions related to $\ker {\Phi}_r$ remain, but these follow at once from the first part of \Cref{lemma:CZB}.

        The properties of ${\Phi}_\infty $ follow from the second part of \Cref{lemma:CZB} and the fact that:
        \[
            \Huf_1(X) = \frac{\bigcup_r \Zsuf_1(R_rX)}{\bigcup_r \Bsuf_1(R_rX)}.
        \]

    \end{proof}

    For later use, we note also that:
    
    \begin{corollary}\label{hsuf_rewritten}
        Let $X$ a ULF graph and assume that $\cC_r(X,A) = \cC_\infty (X,A)$ for some $r>0$.
        Then, for any $s$ large enough, the map ${\Phi}_s$ descends to an isomorphism:
        \[
            \tilde{\Phi }_s :  \Zsuf_1(X,A)/\cC_r(X,A) \rightarrow  \Hsuf_1(R_sX,A),
        \]
        and so does ${\Phi}_\infty $:
        \[
            \tilde{\Phi }_\infty  :  \Zsuf_1(X,A)/\cC_r(X,A) \rightarrow  \Huf_1(X,A).
        \]
        Furthermore, the following diagram (of isomorphisms) commutes:
        \[\begin{tikzcd}
                                                    & \frac{\Zsuf_1(X,A)}{\cC_r(X,A)} \ar[dl,"{\tilde{\Phi }_s}"] \ar[dr,"{\tilde{\Phi }_\infty }"]    \\ 
            \Hsuf_1(R_sX,A)  \ar[rr,"{[\iota ]}"]     &                                                   &\Huf_1(X,A) 
        \end{tikzcd}\]
    \end{corollary}
    \begin{proof}
        Drop the $A$s.
        Consider the maps ${\Phi}_s,{\Phi}_\infty $ of~\Cref{H1uf_rewritten}.
        Since for all $s\geq r$, $\cC_\infty (X) = \cC_s(X) = \cC_r(X)$, the equations involving the respective kernels give us the desired isomorphisms.
        Commutativity of the triangles follows, since all maps involved are (appropriate quotients of) inclusions.
    \end{proof}

%% file: Trees_and_End_Defining_Trees.tex

\section{Trees and End-Defining Trees}\label{section:trees}

Expanding a result of~\cite{Dianathesis}, we describe the first homology of any (ULF) tree $T$ as $l^\infty (B)$, for some set $B$ of bips in $T$, of cardinality the number of ends of $t$ (minus $1$ in the finite case).

We will then show that the first homology of a sufficiently well-behaved subtree of a graph $X$ injects into the first homology of $X$.

From now on, by a trivalent tree, we mean a tree with vertices of degree either $2$ or $3$.
We first verify that in terms of uniformly finite homology, leaves (that is, vertices of degree $1$) can be safely forgotten. 

\begin{lemma}
    Take $A= \mathbb{Z},\mathbb{Z}_2$; then $\Hsuf_1(T,A) = \Hsuf_1(T',A)$ where $T'$ is the maximal subtree of $T$ without leaves (i.e. vertices of degree $1$). 
\end{lemma}
\begin{proof}
    Since $T$ has no non-trivial circuits, we have that $\Hsuf_1(T,A) = \Zsuf_1(T,A)$ (\Cref{H1uf_rewritten}).
    Recall than an element $c\in \Zsuf_1(T,A)$ is a flow on $T$, and since leaves correspond to dead-end for the flow, the edges adjacent to leaves must have zero flow.
    Arguing inductively implies then that $c$ must have support in $T'$, and the equality is achieved.
\end{proof}

    To reduce to trivalent trees, it is now enough to note that any ULF tree is quasi-isometric to trivalent ones.
    Indeed suppose such a tree $T$ is given.
    A trivalent “unfolding” of $T$ can be obtained by replacing each star around a vertex as in~\Cref{startocomb}.
    The uniform bound on degrees ensures that the result, $\tilde T$, is quasi-isometric to $T$.
    Since $\Huf_1$ is QI-invariant, it is therefore sufficient to study trivalent trees.

    \begin{figure}[h]
    \centering
    \includegraphics[width=0.8\textwidth]{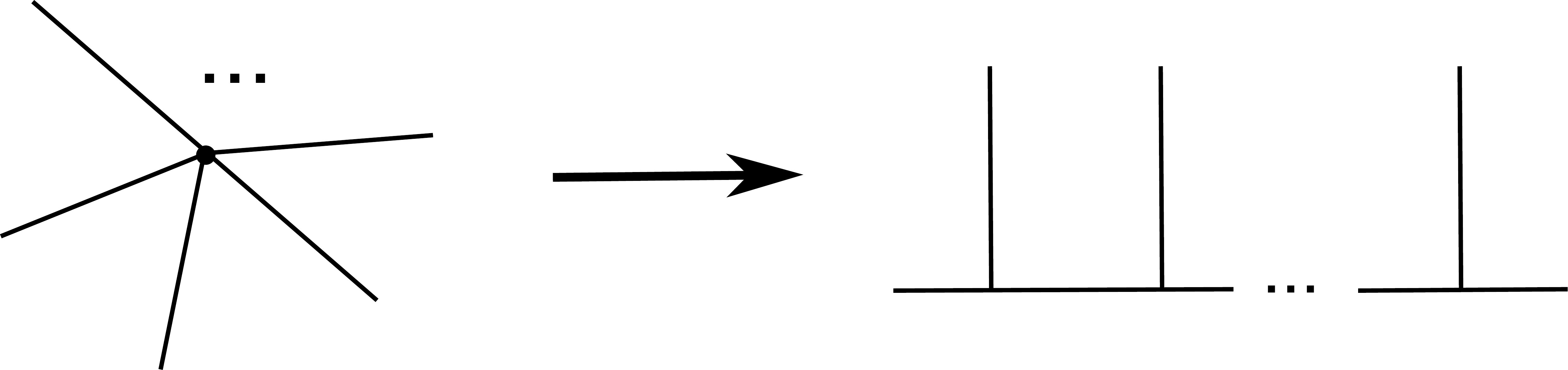}
    \caption{Star to Comb}
    \label{startocomb}
    \end{figure}

\subsection{Homology of Trees}

We want to better describe the first uniformly finite homology of trees.
Recall that we essentially have two kinds of flows on a graph $X$: circuits and bips.
In case $X$ is a tree, circuits are never reduced, thus define zero flow.
The “building blocks” for a flow on $X$ are therefore bips. 
We will extract from $X$ a specific set $P$ of bips which suffices to fully describe flows on $X$.
The construction of $P$ will be explicit but only some abstract “tameness” properties of $P$ will actually be used.


    In order to ease notation we consider the following construction of a trivalent tree (with no leaves).
    \begin{itemize}
        \item
            Start with a single vertex $v_0$, and call that graph $T_0$.
        \item
            Let $T_1$ be the tree corresponding to $\ZZ $, containing $T_0$ as $v_0 = 0$.
            $T_1$ can be viewed as two copies of $\NN $ glued on $v_0$.
    \end{itemize}
    And now inductively:
    \begin{itemize}
        \item
            In $T_n$, choose a set $B_n$ of vertices that does not intersect $T_{n-1}$; call them \emph{branching points} of depth $n$.
            To each $v\in B_n$, glue a copy called $R_v$ of $\NN $ at zero (a new branch).
            For $i\in \NN $, write $i_v$ for the vertex $i$ of $R_v\simeq \NN $, and ${\infty}_v$ for its end (this is not a vertex).
            Let $T_{n+1}$ be the result.
            Formally:
            \[
                T_{n+1} = (T_n \sqcup  \bigsqcup_{v\in B_n} R_v) / \{0_v \sim  v\ |\ v\in  B_n\}. 
            \]
    \end{itemize}
    
    Obviously, $T_n \leq  T_{n+1}$.

    Let us call each $R_v$ the \emph{ray} at $v$ and call $n$ its depth if it is in $T_{n+1}-T_n$.
    The construction $(T_n,B_n)_{n\in\mathbb{N}}$ has depth $N$ if this construction ends at step $T_N$, i.e. if for $k>N$ we have $B_k=\emptyset $.
    The construction has infinite depth otherwise: $T = {\bigcup}_n T_n$.

    Note that for each $n$, the set $B_n$ of branching points of depth $n$ can be partitioned as:
    \[
        B_n = \bigsqcup_{v\in B_{n-1}} B_n \cap  (R_v-\{v\}),
    \]
    since each new branching point has to lie on $T_n - T_{n-1} = \bigsqcup_{v\in B_{n-1}} R_v - \{v\}$.
    Let $B_{n,v} := B_n \cap  R_v - \{v\} = B_n \cap  R_v$ denote the branching points lying on $R_v$.
    There is an obvious ordering on $B_{n,v}$ as a subset of $\NN $: write $v_{v,i}$ the $i$-th element for this ordering (where $i<\alpha $, and $\alpha \in \NN \cup \{\infty \}$ is the cardinal of $B_{n,v}$).
   

    \begin{proposition}
        All (infinite) trivalent leafless trees are constructed in this fashion, modulo one edge.
    \end{proposition}
    \begin{proof}
        Start with a vertex $v_0\in T$ of degree two (if such a vertex does not exist, split an edge to obtain one).
        Follow any two paths starting from $v_0$: this defines $T_1$.
        Let $B_1 := \{v\in T_1:\ \mathrm{deg}(v)>2\}$, and from each $v\in B_1$, follow any new ray to get $T_2$.
        Proceed inductively like that, and let $T_\omega  := {\bigcup}_n T_i$.
        It remains to see that $T_\omega =T$, but this follows from the fact that if $v\in VT$ is at distance $d$ from $v$, then necessarily $v\in T_d$.
    \end{proof}

    We now describe how to construct the set $P$ of bips inductively, by following the construction of $T$.
    Fix $T$ a trivalent tree, and consider a sequence $(T_n,B_n)_{n\in \NN }$ that matches the construction above.
    We construct $P$ by induction: $P_{n+1}$ extends $P_n$ and is a tame set of bips for $T_{n+1}$:

    \begin{description}
        \item[$T_0$:]
            no bip: $P_0 = \emptyset $.
        \item[$T_1$:]
            let $P_1$ contain the unique bip in $P$\ : $P_1 = \{T_1\}$.
    
        \item[$T_{n+1}$:]
            Assume now $P_i$ has been constructed for $i\leq n$.

            We will define $P_{n+1} = P_n \sqcup  P_{n+1}'$, with $P_{n+1}'$ the set of \emph{new} bips, as follows:

            First fix $u\in B_{n-1}$, and consider the ray $R_u$ in $T_n$.
            We recall that the branching points of $B_{n+1}$ can be partitioned by the rays on which they lie.
            Likewise, we define new bips for each ray of depth $n$:
            \[
                P_{n+1,u}' :=
                    \begin{cases}
                        \{R_{v_{u,i}}^{-1}+ [v_{u,i},v_{u,i+1}]+R_{v_{u,i+1}}\ |\  i\in \NN \} & \text{ if $|B_{n,u}| = \infty $}\\
                        \{(R_{v_{u,i}}^{-1}+ [v_{u,i},v_{u,i+1}])+R_{v_{u,i+1}}\ |\ i<k\} \cup  \{R_{v_{u,k}}^{-1}+[v_{u,k},{\infty}_u]\} & \text{ if $|B_{n,u}| = k+1$}
                    \end{cases}
            \]
            where $(v_{u,i})_{i<|B_{n,u}|}$ is the ordering of $B_{n,u}$.
            And then
            \[
                P_{n+1}' := \bigsqcup_{u\in B_{n-1}} P_{n+1,u}', \qquad P_{n+1} := P_n \sqcup  P_{n+1}'.
            \]
    \end{description}
   
    If $(T_n,B_n)_{n\in \NN }$ has infinite depth, let $P = \bigcup_n P_n$.
    We then define a (partial) order on $P$ as follows.
    If $p\in P_n$ and $q\in P_{n+1}$ share an edge, let $p\leq q$, and take the transitive closure.
    Say that $p\in P$ is the \emph{last} bip over an edge $e$ if $p$ lies over $e$ and for any other $q\in P$ lying over $e$, we have $q<p$.
    Write 
    \[
        M_p := \{e \in  ET: \ \text{$p$ is last over $e$}\}.
    \]


    We can now describe the properties of $P$ we will use later on:

    \begin{proposition}[$P$ is tame] \label{P_is_tame}
        Let $(T_n,B_n)_{n\in \NN }$ be a tree construction as above, and $(P_n)_{n\in \NN }$ the corresponding construction of bips.
        Let also $\leq $ be the partial order on $P$ described above.
        \begin{enumerate}
            \item
                Each edge is covered by at most three bips.
            \item
                All elements of $P_n'$ lie in $T_{n+1}-T_{n-1}^{+}$ (where $T_k^{+}$ is the $1$-neighbourhood of $T_k$).
            \item
                $P_n$ covers $T_n$.
            \item
                Each $p\in P$ has finitely many predecessors (with respect to $\leq $), and is the last bip over at least one edge (i.e. $M_p \neq  \emptyset $).
            \item
                For any edge $e$ where $p\in P_{n+1}$ is last over $e$, there exists $e'\in T_n^{+}$ such that:
                \[
                    \{q\in P\ :\ e'\in q\} = \{q\in P\ :\ e\in q<p\} = \{q\in P_n\ :\ e\in q\} =: Q_p.
                \]
            \item
                Each set $M_p$ is connected.
        \end{enumerate}
    \end{proposition}
    \begin{proof}
      By construction.
    \end{proof}

    We can actually endow $P$ with a well-ordering (extending $\leq $), which will facilitate the proof in~\Cref{theorembips}.
    \begin{lemma}\label{poset2total} 
        If $(P,\leq )$ is a countable poset such that $\{\leq p\}$ is finite for all $p$, there exist a well ordering $\preccurlyeq $ on $P$ extending $\leq $ (i.e. $p\leq q\Rightarrow p\preccurlyeq q$) and such that $\{\preccurlyeq p\}$ is still finite for any $p$.
    \end{lemma}
    \begin{proof}
        Define
        \begin{alignat*}{10}
            P_0      &= \left\{p\in P\ :\ \{<p\} = \emptyset \right\}\\
            P_{n+1} &= \left\{p\in P\ :\ \{<p\} \subseteq  \bigcup_{j=1}^n P_j\right\}
        \end{alignat*}
        and note that, assuming $P$ is non-empty:
        \begin{itemize}
            \item
                $P_0\neq \emptyset $ and $P = \bigcup P_j$, since $\{\leq p\}$ is finite for all $p$;
            \item
                The elements of $P_j$ are all pairwise incomparable (endow them all of any well-ordering)
            \item
                Each $P_j$ is countable (since $P$ is).
        \end{itemize}
        For definiteness, assume that $P$ is infinite countable (the finite case is treated similarly).
        It remains then to construct a bijection $\sigma :\NN\rightarrow P$ such that $\sigma n<\sigma m \Rightarrow  n<m$.
           
        The algorithm in \Cref{algo:total_ordering} constructs $\sigma $.
        The idea is to “diagonally” (as in the usual bijection $\NN^{2} \leftrightarrow  \NN$) exhaust all $P_i$s, while preserving the existing order.
        \begin{algorithm} \caption{Extending partial order to total order.} \label{algo:total_ordering}
        \begin{algorithmic}[1]
        \State $i \gets 0,\ \ j \gets 0,\ \ k \gets 0, \ \ \sigma  \gets \emptyset$
        \While{ $\exists  p \in  P_j - \im \sigma  \text{ such that }\ \forall q<p: q\in  \im \sigma $ }
            \State $\sigma  \gets \sigma  \cup  \{i\mapsto p\}$
            \State $i \gets i+1$
            \If {$j\leq k$}
                \State $j \gets j+1$
            \Else
                \State $j \gets 0,\ \ k \gets k+1$
            \EndIf
        \EndWhile
        \end{algorithmic}
        \end{algorithm}

    \end{proof}

    \begin{corollary}\label{posetistotal}
        The set of bips $P$ can be assumed to be well ordered without breaking any of the properties in~\Cref{P_is_tame} (that is, $P$ stays “tame”). 
        In particular, any edge has a last bip on it:
        \[
            ET = \bigsqcup_p M_p
        \]
    \end{corollary}

    %
    %

    We prove that if $T$ is a tree (with uniformly bounded vertex degree) and $P$ is as constructed above, then we can use it to describe $\Huf_1(X)$.
    We invite the reader to have a quick peak at the theorem.
    Before proceeding with the statement, let us first motivate the tameness properties of $P$.
    The imposed order and the finiteness condition help in making an inductive argument.
    The non-emptiness of the sets $M_p$ guarantees that the chosen bips are independent, that is, there is no bounded sum that is zero in homology.
    Once we have independence of bips, the connectedness of $M_p$ will guarantee generation of all homology classes.
    The reader might have the feeling that we are constructing some kind of infinite basis — this would be correct, as the theorem states.
    Finally the two remaining conditions just make sure that we only have to consider uniformly bounded sums of bips.

    \begin{theorem}\label{theorembips}
        Given a trivalent tree $T$, a set $P$ of bips on $T$ as constructed above and $A\in \{\ZZ,\ZZ_2\}$, the following map:
        \begin{alignat*}{10}
            l^\infty (P)   &\rightarrow  \Hsuf_1(T,A) \cong  \Huf_1(T,A)\\
            f           &\mapsto  c_f := [e\mapsto \sum_{e\in p\in P}f(p)]
        \end{alignat*}
        is an isomorphism  (continuous when $A=\ZZ$ and $\Hsuf_1(T,A)$ has the $l^\infty $ norm).
    \end{theorem}
    \begin{proof}

        By~\Cref{posetistotal}, we can assume that $P$ is well-ordered.

        The map above is easily seen to be linear, and of norm $\leq K$ (where $|P_e|\leq K$ for all $e\in E$).
        \point{Injectivity:} By induction.
        If $f\neq 0 \in  l^\infty (P)$, there exists a bip $p$ such that $f(p)\neq 0$; take a least such $p$.
        Consider then $e\in M_p$:
        \begin{alignat*}{10}
            c_f(e) = \sum_{e\in q\in P} f(q) &= \sum_{e\in q<p} f(q) + f(p)\\
        \end{alignat*}
        Since by assumption, for each $e\in q<p$, $f(q)=0$, it follows that $c_f(e)\neq 0$.

        \point{Surjectivity:}
        Take a cycle $c \in  \Zsuf_1(T,A)$.
        Then, define inductively for $p\in P$:
        \[
            f(p):= c(e) - \sum_{e\in q<p} f(q),
        \]
        where $e\in M_p$.
        We check that this is independent of the choice of $e\in M_p$.
        By connectedness of $M_p$, it is enough to verify that adjacent edges $e_1,e_2$ in $M_p$ yield the same value.
        There are only two possible configurations locally:
        
        \begin{minipage}[b]{0.5\textwidth}
        \begin{tikzpicture}
        \filldraw[double] (-0.5,0) -- (0,0) circle (2pt) node[align=left,   below]{}  -- node[above] {$e_1$}
        (2,0) circle (2pt) node[align=center, below] {} -- node[above] {$e_2$}
        (4,0) circle (2pt) node[align=right,  below] {} -- (4.5,0);
        \draw[green] (0,-0.2) -- (4,-0.2); 
        \end{tikzpicture}
        \end{minipage}
        \begin{minipage}[b]{0.5\textwidth}
        \begin{tikzpicture}
        \filldraw[double] (-0.5,0) -- (0,0) circle (2pt) node[align=left,   below]{}  -- node[above] {$e_1$}
        (2,0) circle (2pt) node[align=center, below] {} -- node[above] {$e_2$}
        (4,0) circle (2pt) node[align=right,  below] {} -- (4.5,0);
        \filldraw (2,0)  -- (2,-0.5) --  node[right] {$e_3$}(2,-1) ;
        \draw[green] (0,-0.2) -- (4,-0.2); 
        \draw[red] (0,-0.3) -- (1.9,-0.3) -- (1.9, -1);
        \draw[blue] (4,-0.3) -- (2.1,-0.3) -- (2.1, -1);
        \end{tikzpicture}
        \end{minipage}

        Horizontally one can see the bip $p$, for which we are defining the value $f(p)$.
        The green ($G$), red ($R$) and blue ($B$) lines represent groups of bips running over these vertices (possibly empty).
        Note that all these bips are in order strictly smaller than $p$ (because $p$ monopolises $e_1,e_2$).
        Hence $f$ is defined for all of them.
        We check that $f(p)$ does not depend on whether we choose $e_1$ or $e_2$.
        There are two situations possible depicted above.
        For the first picture the claim clearly holds, we check the second picture:
        \begin{align*}
            c(e_1) - \sum_{q\in G \text{ or } q\in R} f(q) &= c(e_2) + c(e_3)  - \sum_{q\in G \text{ or } q\in R} f(q)\\
            &= c(e_2) +   \sum_{q\in R } f(q) -  \sum_{q\in B } f(q) - \sum_{q\in G \text{ or } q\in R} f(q)\\
            &= c(e_2) - \sum_{q\in G \text{ or } q\in B} f(q)
        \end{align*}  
        The second equality holds by the induction argument.
        
        Now $f$ is uniformly bounded: from
        \[
            f(p):= c(e) - \sum_{e\in q<p} f(q),
        \]
        we see that 
        \[
            |f(p)| \leq  \|c\|_\infty  + \|\sum_{e\in q<p}f(q)\| \leq  2\|c\|_\infty ,
        \]
        since the elements of $\{e\in q<p\}$ are assumed to share an edge (so that the sum over $\{e\in q<p\}$ is exactly the value $c(e')$ of some edge $e'$).
        
        It is now easily seen that $c_f = c$, and since $\|f\|_\infty  \leq  2\|c\|_\infty $, the desired isomorphism is reached.

    \end{proof}
        
    Note that only some of the abstract “tameness” properties of the set $P$ described in \Cref{P_is_tame}, plus well-ordering, are actually used in the proof.
    Furthermore, one could work without well-ordering (well-foundedness of $\leq $ is enough), at the cost of a less transparent proof.
   
    As an example of necessity of the “tameness” conditions, we give~\Cref{figurebipscounterexample}; the set of bips chosen satisfy each condition but the fifth one — the existence of an edge $e'$ on which exactly the predecessors of a given bip $p$ run.
    As a result the flow defined on the teeth of the comb cannot be described as a bounded sum of the given bips.

    \begin{figure}[h]
        \centering
        \begin{tikzpicture}

            \filldraw (0,0) -- (1,0) circle (2pt) node[align=left,   below]{}  -- (3,0) circle (2pt) node[align=center, below] {} -- (5,0) circle (2pt) node[align=center, below] {} -- (7,0) circle (2pt) node[align=center, below] {} -- (9,0) circle (2pt) node[align=center, below] {} -- (11,0) circle (2pt) node[align=center, below] {} -- (13,0) circle (2pt) node[align=center, below] {} -- (15,0) circle (2pt) node[align=center, below] {}-- (16,0)   ;
            \draw (1,0) -- (1,3) node[align=center, above] {$1$}; 
            \draw (3,0) -- (3,3) node[align=center, above] {$-1$}; 
            \draw (5,0) -- (5,3) node[align=center, above] {$1$}; 
            \draw (7,0) -- (7,3) node[align=center, above] {$-1$}; 
            \draw (9,0) -- (9,3) node[align=center, above] {$1$}; 
            \draw (11,0) -- (11,3) node[align=center, above] {$-1$};
            \draw (13,0) -- (13,3) node[align=center, above] {$1$};
            \draw (15,0) -- (15,3) node[align=center, above] {$-1$};
            \draw[purple] (0,-0.2) -- (16,-0.2) ; 
            \draw[green] (0,0.1) -- (0.9,0.1) -- (0.9,3); 
            \draw[green] (1.1,3) -- (1.1,0.1) -- (4.9,0.1) -- (4.9,3); 
            \draw[green] (9.1,3) -- (9.1,0.1) -- (12.9,0.1) -- (12.9,3); 
            \draw[green] (13.1,3) -- (13.1,0.1) -- (16,0.1); 
            \draw[red] (0,0.2) -- (2.9,0.2) -- (2.9,3); 
            \draw[red] (3.1,3) -- (3.1,0.2) -- (6.9,0.2) -- (6.9,3);  
            \draw[red] (11.1,3) -- (11.1,0.2) -- (14.9,0.2)--(14.9,3); 
            \draw[red] (15.1,3) -- (15.1,0.2) -- (16,0.2); 
        \end{tikzpicture}
        \caption{
            A bi-infinite comb with bips as shown in green and red.
            The cycle we consider is defined by alternating positive and negative values on the teeth of the comb.
            The reader can check that the obtained function on the bips is unbounded.
        }\label{figurebipscounterexample}
    \end{figure}
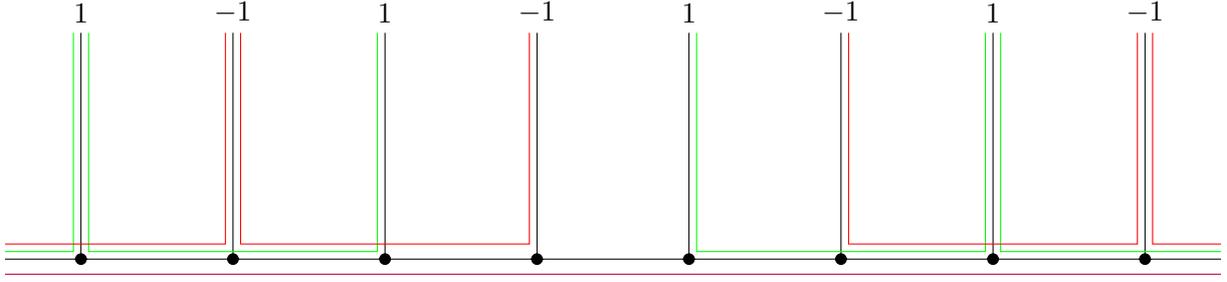
            

\subsection{Ends}

    Now that the homology of trees is understood, we can relate it to homology of arbitrary graphs via their ends.
    The core idea is that, when working with ends of graphs, working with nicely embedded trees is enough.

    We start with a definition:
    \begin{definition}[{End-respecting (defining) trees}]
        Let $X$ be an arbitrary ULF graph.
        An end-respecting (resp. end-defining) tree for $X$ is a leafless tree $T\leq X$ satisfying:
        Any end of $X$ is represented by at most one (resp. a unique) infinite branch of $T$.
        Furthermore, $T$ is said to \emph{have separators} if: 
        There exists a sequence $(K_n)_n$ of increasing finite, connected subsets of $VX$ satisfying:
        \begin{itemize}
            \item
                $T\cap K_n$ is connected for all $n$;
            \item
                $\bigcup_n K_n = VX$; 
            \item
                For each infinite connected component $L$ of $X-K_n$, there is exactly one edge of ${\partial}_T(T\cap K_n)$ \footnote{${\partial}_X Y = $ edges of $X$ with exactly one endpoint in $Y$.} incident to $L$, and vice versa.
                This defines a bijection between infinite connected components of $X-K_n$ and ${\partial}_T(T\cap K_n)$.
        \end{itemize}
        Fix $v_T$ a root for $T$, chosen in $K_0$.
    \end{definition}

    \begin{remark}

    End-defining trees with separators can always be constructed.
    Indeed, fix $X$ and some $x\in VX$.
    Consider the smallest $R>0$ such that removing $B(x,R)$ from $X$ disconnects the graph in multiple infinite components.
    Now choose as many edges in the boundary of the ball as there are infinite components, and a geodesic for each of these starting in $x$.
    Combining these geodesics gives a subgraph, and deleting edges as necessary yields a tree.
    Now one repeats this approach inductively:
    Suppose we have a finite tree defined up to $B(x,R)$, with boundary edges $l_i$ in correspondence to infinite connected components $L_i$ (of $X-B(x,R)$).
    Take $\tilde R>R$ least such that $X-B(x,\tilde R)$ has more components than $X-B(x,R)$.
    If a connected component $L_I$ of $X-B(x,R)$ is further separated in components $L_I^{j}$ in $X-B(x,\tilde R)$, with $j\in\{ 1\dots k \}$ with distinguished boundary edges $l_I^j$ for each component in the boundary of $B(x,\tilde R)$, then as before we take geodesics from $l_I$ to each $l_I^j$ and delete possible cycles.
    
    On the other hand, it is not clear to us whether one can always find separators for end-respecting trees. 
    \end{remark}

    The following technical result is close to~\cite[Lemma 2.4]{BW}.
    Given $\phi ,\psi :EX\rightarrow \ZZ$, say that $\psi \leq \phi $ if, for any $e\in EX$: $\phi e\geq 0$ implies $0\leq \psi e\leq \phi e$ and $\phi e\leq 0$ implies $\phi e\leq \psi e\leq 0$.
    \begin{lemma}\label{lemma:extending_rays_generalised}
        Let $\phi \in \Zsuf_1(X,\ZZ )$ and $r:EX\rightarrow  \ZZ $ with $r\leq \phi $, $r\notin \Zsuf_1(X,\ZZ )$, and $\partial r\geq 0$
        Then there exists $\hat r \in \Zsuf_1(X,\ZZ )$ satisfying $r \leq  \hat r \leq  \phi $.
        Furthermore, $\hat r$ and $r$ are different on infinitely many edges.
    \end{lemma}

    Here is a picturesque interpretation:
    View $\phi $ as closed (in the sense that in$=$out) flow on the graph $X$.
    Then $r$ is a non-closed flow that's no greater than $\phi $: any flow on the edges is in the same direction as $\phi $, and not stronger.
    Then $r$ can be made into a closed flow by simply adding flow to edges in the directions allowed by $\phi $ (so, no reduction of what $r$ had and not exceeding what $\phi $ allows).


    \begin{proof}
        We extend $r$ step by step, and see that the sequence of extensions converges:
        \begin{enumerate}
            \item
                For any $v$ with $\partial r(v)>0$, we have:
                \[
                    \sum_{u\sim v} \phi (u,v) = 0,
                \]
                so, there must exist $e_v=(u,v)$ with $\phi (e_v)<r_v(e)\leq 0$.
            \item
                Consider now $r^1 := r + \sum_{v\in V_0} e_v$, where $V_0$ is the set of $v$s with $\partial r(v)>0$.
                We notice that this sum is well-defined since no $e_v$ appears in two distinct $v$s, and $r\leq r^{1}\leq \phi $ by construction.
                Also, $\partial r^{1}\geq 0$ since if $v\in V_0$, then $v$ stays (non-strictly) positive in $\partial r^{1}$, and the other $v$s can only increase, by the direction of the added edges.
            \item
                Redo the argument with $r$ replaced by $r^{1}$, and construct inductively a sequence $r=r^{0}\leq r^{1}\leq \ldots\leq \phi $.
            \item
                Since the sequence $(r^i)_i$ is monotonous and  “below” $\phi $, it must converge to some $r \leq  \hat r \leq  \phi $, in $\Zsuf_1(X,\ZZ )$, by boundedness of $\phi $.
            \item
                The statement that $r$ and $\hat r$ differ on an infinite set can be deduced from the fact that the process of extension $r^i \leadsto r^{i+1}$ cannot end after finitely many enumerate.
                Indeed, if $v\in V_i$, and $(u,v)=e_v$ is the edge chosen to “cancel” $v$, then $u$ becomes positive in $\partial r^{i+1}$, so that there are always positive edges.
        \end{enumerate}
    \end{proof}

    A similar result holds for $\ZZ_2$.
    In $EX\rightarrow \ZZ_2$, say $r \leq  \phi $ iff $\phi (e) = 0$ implies $r(e) = 0$ for all $e\in EX$.
    
    \begin{lemma}\label{ray_to_bip_Z2}
        Let $X$ a (ULF) graph, and $\phi  \in  \Zsuf_1(X,\ZZ_2)$, $r \in  \Csuf_1(X,\ZZ_2)$ with $r \leq  \phi $, and $\partial r = {\delta}_v$ for some vertex $v$.
        Then, there exists $r \leq  \hat r \leq  \phi $ with $\hat r \in  \Zsuf_1(X,\ZZ_2)$ and such that $\hat r$ differs from $r$ on infinitely many edges. 
    \end{lemma}
    \begin{proof}
        By mimicking the $\ZZ$ case.
        Since $\sum_{e \sim v} \phi (e) = 0 \neq  \sum_{e \sim v} r(e)$, there exists $e$ with $\phi (e)=1, r(e) = 0$.
        Thus, define $r^{1} := r + e$ and continue by induction.
    \end{proof}

	We can now prove the following. 

    \begin{proposition}[End-respecting trees inject in $\Huf_1$]\label{tree_injection}
        Let $X$ be ULF and $T$ an end-respecting tree with separators.
        Then:
        \[
            \Huf_1(T,\ZZ ) \leq  \Huf_1(X,\ZZ ).
        \]
        The same holds for $\ZZ_2$.
    \end{proposition}
    \begin{proof}
        We first consider $\ZZ$.
        Let $(K_i)_i$ the sequence of “separators” associated to $T$.
        \begin{enumerate}
            \item
                First, recall that $\Huf_1(X,\ZZ ) \cong  \Zsuf_1(X,\ZZ )/\cC_\infty $, so, in particular, $\Huf_1(T,\ZZ ) \cong  \Zsuf_1(T,\ZZ )$, $T$ being a tree.
                From now on, drop the $\ZZ $.
                We consider the function:
                \[
                    \Zsuf_1(T) \ovs{\subseteq }\rightarrow  \Zsuf_1(X) \ovs{/}\rightarrow  \Zsuf_1(X)/\cC_{\infty}
                \]
                so that showing injectivity amounts to showing that no non-zero function in $\Zsuf_1(T)$ can be written as an element of $\cC_\infty $.
                We actually show more: the intersection of $\Zsuf_1(T)$ and $\cC$ is trivial.
            \item
                Assume, towards a contradiction, that there exists $0\neq f\in \Zsuf_1(T) \cap  \cC$.
                In particular there exists a thin set of circuits $C$ such that:
                \[
                    f = \sum C.
                \]
                (Formally, one should add integer coefficients to the sum, but we'll stick to the shorthand).
            \item
                Since $T$ is end-respecting, there exists some $K=K_n$ and an edge $e\in \supp f$ so that $K$ disconnects the space in multiple infinite components, one of which -- call it $L$ -- is such that $\supp f \cap  \partial L$, consists of exactly one edge $e$.
            \item
                Up to taking $-f$, we may assume that $e$ is directed “away from $L$”, and that $f$ is positive on $e$.
            \item \label{follow_f_backwards}
                Following $f$ backwards starting from $e$, we get a set of $k:=f(e)$ rays $r_1,\ldots,r_{k}$ supported in $T\cap (L\cup \{e\})$and all ending with $e$.
            \item
                Let $C_L$ the set of circuits in $C$ that lie in $L$ and do not intersect $K$.
                
                Let $r$ be the sum of initial segments of each $r_i$, so that $r$ has support not intersecting the circuits of $C$ not in $C_L$ (only a finite number of them have to be taken into account since $C$ is thin and $K$ finite).
                Note that $r\leq f$ by construction.

                Let $f_L$ be the restriction of $f-(\sum r_i)$ to $L$.
            \item
                Remark that $f_L \in  \Zsuf_1(L)$, $r\leq C_L - f_L$ and, far enough from $K$, we have:
                \[
                    r = C_L - f_L
                \]
            \item \label{special_Z}
                By \Cref{lemma:extending_rays_generalised}, there exists $\hat r\in \Zsuf_1(L)$ with $r\leq \hat r \leq  C_L - f_L$, $\hat r$ distinct from $r$ at infinitely many edges.
            \item
                Then $\hat r$ must be distinct from $r$ away from $K$.
                Since $r\leq \hat r \leq  C_L - f_L$ and, away from $K$, $r=C_L - f_L$, we reach a contradiction.
        \end{enumerate}

        The case of $\ZZ_2$ is a simple adaptation.
        We follow the same procedure up to step~\ref{follow_f_backwards}.
        There, we can just follow \emph{one} ray $r_1$ along $f$.
        Then, at step~\ref{special_Z}, we use \Cref{ray_to_bip_Z2} instead of \Cref{lemma:extending_rays_generalised} to extend $r$ to some bip $\hat r$, with $\hat r \leq  \sum C_L - f_L$, and differing from $r$ arbitrarily far from $K$, from which we again deduce a contradiction.

    \end{proof}

    We extract the core property used in the proof for future use:
    \begin{corollary}[of the proof]\label{Z1T_and_Cycle_is_trivial}
        With the assumptions of~\Cref{tree_injection}:
        \[
            \Zsuf_1(T) \cap  \cC(X) = 0.
        \]
    \end{corollary}

    \begin{corollary}
        The dimension of $\Huf_1(X,\ZZ)$ is at least the number of ends of $X$ minus $1$.
    \end{corollary}

    Recall that any finite subset in a graph induces a partition of its ends given by the connected component in the complement of the set in which they lie.
	\begin{lemma}\label{boundaryertrees}
        \newcommand{\cP}{\mathcal{P}}
        Let $T_1,T_2$ be two end-defining trees in the ULF graph $X$.
        For any edge $e$ in $T_1$, there exists a compact $K\subset X$, such that the partition $\mathcal{E}$ induced by $K$ on the ends of $T_2$ is a subpartition of the one induced by $e$ on $T_1$. 
        In particular, the shadow topology induced on the $\partial X = \partial T_1 =\partial T_2$ by $T_1$ is the same as the topology induced by $T_2$. 
	\end{lemma}

	\begin{proof}
        \newcommand{\cP}{\mathcal{P}}
	    Write $\{E^{+},E^{-}\}$ for the partition of ends induced by $e$ in $T_1$.
        Fix a vertex $x$ on $T_1$ and $x_2$ on $T_2$
        Suppose for every ball $B_R$ with radius $R$ and center $x$, the partition of ends induced by it on $T_2$, indirectly given by the connected components of $T_2- B$, is not a subpartition, then there are ends $\varepsilon_i$, $\varepsilon_i'$ such that in $T_2- B_i$ they are in the same connected component, but $\varepsilon_i\in E^-$ and $\varepsilon_i'\in E^+$
        Now take rays $r_i, r_i'$ starting at $x$ (resp. $\tilde{r}_i$, $\tilde{r}_i'$ starting at $x_2$) in $T_1$ (resp. $T_2$), such that their limits are the ends $\varepsilon_i$ and $\varepsilon_i'$
        Note that since $\tilde{r}_i$ and $\tilde{r}_i'$ are in the same connected component, they overlap on their first vertices inside $B_i$
        This means that they have the same pointwise limit
        On the other hand consider $r_i$ and $\tilde{r}_i$
        Take any end $\varepsilon $ of $X$, then it corresponds to a choice of connected components of $X$ after taking away bigger and bigger compact sets (say for example the balls $B_R$)
        However since $r_i(\infty) = \tilde{r}_i(\infty)$, they are in the same connected component if one goes sufficiently far
        However this means that any consistent set of pointwise convergence accumulation points of $\{r_i\}_i$ and $\{\tilde{r}_i\}_i$ have the same end
        After restricting to a subsequence we can assume that both sequences converge pointwise
        Then, because of previous arguments, the end of $(\lim r_i)(\infty) = (\lim \tilde{r}_i)(\infty)$, which is equal to $(\lim \tilde{r}_i')(\infty)$  and thus equal to $(\lim r_i')(\infty)$
        However this is a contradiction, since the ends of $r_i$ are in $E^-$ and those of $r_i'$ are in $E^+$
        Since both kinds of infinite rays are separated by the edge $e$, also the ends of the limit rays are in these respective sets.
	\end{proof}

    The result in~\Cref{tree_injection} can actually be generalised to end-respecting trees \emph{without} separators. 
    \begin{proposition}
        Let $X$ be ULF and $T\leq X$ an end-respecting tree.
        Then:
        \[
            \Huf_1(T,\ZZ ) \leq  \Huf_1(X,\ZZ ).
        \]
        The same holds for $\ZZ_2$.
    \end{proposition}
    \begin{proof}
        \newcommand{\tT}{\tilde T}
            \renewcommand{\tf}{\tilde f}
                \newcommand{\tb}{\tilde b}
                    \newcommand{\tP}{\tilde P}
                        \newcommand{\cP}{\mathcal{P}}
        
        Let $T\leq X$ such a tree, and $\tT$ an end-defining tree for $X$ with separators.
        \begin{itemize}
            \item
                We show that $\Hsuf_1(T,\ZZ)\cap \cC(X,\ZZ)=\{0\}$, and then the argument in the first step of the proof of~\Cref{tree_injection} gives the desired embedding.
            \item
                Towards a contradiction, assume that $f\in \Hsuf_1(T,\ZZ)\cap \cC$ is non-zero.
                By the decomposition of $\Hsuf_1(T,\ZZ)$ into $l^\infty (P)$ for a set of bips $P$, one can write
                \[
                    f = \sum_{b\in P} a_b\cdot b.
                \]
            \item
                For any bip $b\in P$, one can consider its “parallel” bip $\tb$ in the tree $\tT$; that is, $b$ has two ends corresponding to ends of the graph $X$, which correspond also to two ends in $\tT$ (since $\tT$ is end-defining), and finally to a unique bip.
                Let $\tP := \{\tb\ |\ b \in  P\}$ and 
                \[
                    \tf = \sum_{b\in P} a_b\cdot \tb.
                \]
                Note that by “projecting”, one can find a collection of circuits $C_f$ (and coefficients $a_c$) such that 
                \[
                    \tf = f + \sum_{c\in C_f} a_c\cdot c
                \]
                (the construction for $C_f$ is similar to the one in the proof of~\Cref{decomp_Z2}).
            \item
                We claim that $\tf$ is well-defined (but potentially unbounded) and non-zero.
            \item
                Assume for now that this is the case.
                Since $\tf$ is non-zero, there exists a first separator $K_i$ and an edge $e$ lying on $\partial K$ such that $\tf$ is zero on $K$ and non-zero on $e$.
            \item
                Let $C_f^K$ denote those circuits in $C_f$ that intersect the separator $K_{i+1}$ that follows $K_i$ and
                \[
                    f' := f - \sum_{c\in C_f^K}a_c\cdot c.
                \]
                By construction, $f'$ is equal to $\tf$ inside of $K_{i+1}$, so in particular in $K_i$ and on $e$; hence $f'(e)\neq 0$.
            \item
                Following the argument of~\Cref{tree_injection} starting from step 4 with function $f'$ shows that $f'$ cannot be non-zero, a contradiction.
        \end{itemize}
        It remains to verify that $\tf$ is actually well-defined and non-zero. Well-defined is a consequence of Lemma \ref{boundaryertrees}.

        Assume that $\tf=0$, and let $e$ some edge of $T$ on which $f$ is non-zero.
        \begin{itemize}
            \item
                Let $E^{+}\sqcup E^{-}$ the partition of the ends of $T$ following the side of $e$ on which they lie.
                By Lemma \ref{boundaryertrees}  there exists a large enough separator $K_i$ in $X$ that the induced partition on ends is a subpartition of $\{E^{+},E^{-}\}$
                
                Let then $K$ such a separator and $\{E_1^{+},\ldots,E_n^{+}\}\cup \{E_1^{-},\ldots,E_m^{-}\}$ the subpartition induced by $K$.
                Let also $e_1^{+},\ldots,e_n^{+}, e_1^{-},e_m^{-}$ the edges in the boundary of $K$, with corresponding signs.
                The ends in $E_i^\pm $ are then exactly those in the component of $X-K$ corresponding to $e_i^\pm $.
        \end{itemize}
        For a collection of bips $B$, let $a_B := \sum_{b\in B} a_b$. If we note $b^{-1}$ for the bip $b$ with opposite orientation, then $a_{b^{-1}} =  -a_b$. For a collection of oriented bips $B$ and two disjoint sets of ends $E_1,E_2$, let $[E_1,E_2]$ be $\{ b\in B\;\lvert \; b(-\infty)\in E_1 \text{ and } b(+\infty) \in E_2 \}\sqcup \{ b^{-1} \;\lvert \; b\in B \text{ and } b(-\infty)\in E_2 \text{ and } b(+\infty) \in E_1 \}$. In words, the bips having an end in $E_1$ and in $E_2$, taken with consistent  orientation. 
        \begin{itemize}
            \item
                Since $\tf=0$, we have $\tf(e_i^s)=0$ for all $i$ and $s=\pm $.
                This implies
                \[
                    a_{[E^{+}_i,(E^{+}_i)^\complement ]} = 0
                \]
            \item
                We also have:
                \begin{alignat*}{10}
                    f(e)    &= a_{[E^{-},E^{+}]}, \\
                    \intertext{and using $[E^{-},E^{+}] = \bigsqcup_{i=1}^m [E_i^{-},E^{+}]$:}
                            &= \sum_{i=1}^m a_{[E_i^-,E^+]} \\
                            &= \sum_{i=1}^m a_{[E_i^-,E^+]} + 0, \\
                            \intertext{now, using the fact that $a_{[E_j^{-},E_j^{+}]} = - a_{[E_j^{+},E_j^{-}]}$ for any $j$:}
                            &= \sum_{i=1}^m \left( a_{[E_i^-,E^+]} + \sum_{j=1,j\neq i}^m \overbrace{  a_{[E_i^-,E_j^-]}   + a_{[E_j^-,E_i^-]}}^{=0}\right)\\
                            &= \sum_{i=1}^m \left(\left( a_{[E_i^-,E^+]} + \sum_{j=1,j\neq i}^m  a_{[E_i^-,E_j^-]} \right)  + \sum_{j=1,j\neq i}^ma_{[E_j^-,E_i^-]}\right),\\
                            \intertext{and now by the decomposition $[E_i^{-},(E_i^{-})^\complement ] = [E_i^{-},E_i^{+}] \sqcup  \bigsqcup_{j=1,j\neq i}^m [E_i^{-},E_j^{-}]$ for any $i$:}
                            &= \sum_{i=1}^m \left(\left( a_{[E_i^-,(E_i^{-})^\complement ]} \right)  + \sum_{j=1,j\neq i}^ma_{[E_j^-,E_i^-]}\right)\\
                            \intertext{and the content of the inner parenthesis vanishes since $\tf(e_i^-)=0$:}
                            &= \sum_{i=1}^m \sum_{j=1,j\neq i}^ma_{[E_j^-,E_i^-]}\\
                            &= \sum_{1\leq i\neq j\leq m} a_{[E_j^-,E_i^-]} + a_{[E_i^{-},E_j^{-}]}\\
                            \intertext{and finally, again since inverting directions negates the sign:}
                            &= 0.                            
                \end{alignat*}
            
                It follows therefore that $\tf=0$ implies $f(e)=0$, which is a contradiction.
        \end{itemize}

    \end{proof}

%% file: Coefficients_in_Z2.tex

\section{Coefficients in $\ZZ_2$}\label{section:Z2}

    In this section, we study the special case of $\Huf_1(\cdot ,\ZZ_2)$, or “uniformly finite homology with coefficients in $\ZZ_2$”.
    In that case, the “uniformly finite” condition becomes vacuous, and one may wonder if any insight can be gained with such a restriction.
    This is indeed the case, as the next proposition shows, thus motivating the study of $\Huf_1(\cdot ,\ZZ_2)$.

\subsection{$\Huf_1(X,\ZZ) \twoheadrightarrow  \Huf_1(X,\ZZ_2)$}

    \begin{proposition}
        Let $X$ a (ULF) graph.
        Then $\Huf_1(X,\ZZ_2)$ is a quotient of $\Huf_1(X,\ZZ)$.
    \end{proposition}
    \begin{proof}
        Consider the natural map:
        \[
            \Phi : \Zsuf_1(X,\ZZ) \ovs{q_{/2}}\rightarrow  \Zsuf_1(X,\ZZ_2) \ovs{\pi }\twoheadrightarrow  \Zsuf_1(X,\ZZ_2)/\cC_\infty (X,\ZZ_2).
        \]
        The first arrow is surjective.
        Indeed, if $f:EX\rightarrow \ZZ_2$ lies in $\Zsuf_1(X,\ZZ_2)$, one can “follow paths” and write 
        \[
            f = \sum C + \sum B,
        \]
        for $C$ a set of circuits and $B$ a set of bips, such that the elements of $B\cup C$ have pairwise disjoint support.
        Choosing signs arbitrarily, we can lift $f$ to $\tf \in  \Zsuf_1(X,\ZZ)$ by lifting each element of $C$ and $B$.

        To get a surjection $\Huf_1(X,\ZZ) \twoheadrightarrow  \Huf_1(X,\ZZ_2)$, it remains to check that $\cC_\infty (X,\ZZ)$ is contained in $\ker \Phi $.
        But if $\sum C \in  \cC_\infty (X,\ZZ)$ is given, then the image of $\sum C$ through $q_{/2}$ is in $\cC_\infty (X,\ZZ_2)$.

    \end{proof}
    


\subsection{Decomposition of $\Huf_1(X,\ZZ_2)$}

    The main result of the section is the following:

    \begin{proposition}\label{decomp_Z2}
        Let $X$ be a (ULF) graph, and $T\leq X$ an end-defining subtree.
        Then:
        \[
            \Huf_1(X,\ZZ_2) \cong  \Huf_1(T,\ZZ_2) \oplus  \frac{\cC(X,\ZZ_2)}{\cC_\infty (X,\ZZ_2)}.
        \]
    \end{proposition}

    \begin{proof}
        Drop the $\ZZ_2$.
        Recall that $\Huf_1(X) = \Zsuf_1(X)/\cC_\infty $ and $\Huf_1(T) \cong  \Zsuf_1(T)$.
        Furthermore, both $\Zsuf_1(T)$ and $\cC$ lie in $\Zsuf_1(X)$.
        Since $\cC_\infty  \leq  \cC$, it suffices therefore to show that $\Zsuf_1(T) \cap  \cC = 0$ and $\Zsuf_1(T) + \cC = \Zsuf_1(X)$.
        We already know, by~\Cref{Z1T_and_Cycle_is_trivial} that $\Zsuf_1(T)\cap \cC = 0$.
        It remains therefore to show $\Zsuf_1(T)\cap \cC = \Zsuf_1(X)$.

        Fix $f \in  \Zsuf_1(X)$.
        We will find a thin set of circuits $D$ such that $f - \sum D$ is supported on $T$; in some sense “pushing” $f$ towards $T$ using circuits.

        By “following paths”, one can write:
        \[
            f = \sum_{\varepsilon {\neq}\varepsilon '}\sum_{b\in B_{\varepsilon ,\varepsilon '}}b + \sum_{b\in B_=}b + \sum_{c\in C}c
        \]
        with $B_{\varepsilon ,\varepsilon '}$ a set of bips going from end $\varepsilon $ to end $\varepsilon '$, $B_=$ a set of bips going to the same ends, and $C$ a set of circuit, and so that no edge is shared between elements of $\bigcup_{\varepsilon {\neq}\varepsilon '} B_{\varepsilon ,\varepsilon '} \cup  B_= \cup  C$ -- the sum is therefore (very!) thin.
        
        If $B_{\varepsilon ,\varepsilon '}$ is finite, let, for any $b\in B_{\varepsilon ,\varepsilon '}$, $b^*$ denote the unique bip in $T$ with same ends.
        If $B_{\varepsilon ,\varepsilon '}$ is infinite, let $B_{\varepsilon ,\varepsilon '}^{+}\sqcup B_{\varepsilon ,\varepsilon '}^{-} = B_{\varepsilon ,\varepsilon '}$ be a partition into two infinite parts, and choose a bijection $B_{\varepsilon ,\varepsilon '}^{+} \rightarrow  B_{\varepsilon ,\varepsilon '}^{-}$ denoted as sending $b\in B_{\varepsilon ,\varepsilon '}^{+}$ to $b^* \in  B_{\varepsilon ,\varepsilon '}^{-}$.
        Denote also the inverse of this bijection with ${\cdot}^*$, so that $(b^*)^* = b$.

        Fix some $R_0>0$, and let $B_{\varepsilon ,\varepsilon '}^0,B_=^0, C^0$ be the elements of each set ($B_{\varepsilon ,\varepsilon '},B_=,C$) that intersect $B(v_T,R_0)$.
        By ULF, each of those sets is finite, and only a finite number of pairs $\varepsilon {\neq}\varepsilon '$ have $B_{\varepsilon ,\varepsilon '}^{0}$ non-empty.
        For any $b\in B_{\varepsilon ,\varepsilon '}^{0}$, choose two paths $p_{b,0}^{+},p_{b,0}^{-}$ connecting the two rays of $b-B(v_T,R_0)$ outside of $B(v_T,R_0)$ to $b^*$.
        Similarly, for any $b\in B_=^0$, find a path $p_{b,0}$ joining the two rays of $b$ outside of $B(v_T,R_0)$, say $p_{b,0}$ connects $b(t_{b,0}^{+})$ to $b(t_{b,0}^{-})$.
        Let $R_1$ large enough that $B(v_T,R_1)$ contains each of the path used above and each element of $C^0$, let $B_{\varepsilon ,\varepsilon '}^1,B_=^1,C^1$ be the elements of each set that intersect $B(v_T,R_1)$ and repeat the above.
        
        This defines circuits, say $c_{b}^i$ such that:
        \[
            \sum c_b^i = b - b^*
        \]
        if  $b\in B_{\varepsilon ,\varepsilon '}$, and
        \[
            \sum c_b^i = b,
        \]
        if $b\in B_=$.
        Furthermore, at each step $i$, the circuits added do not intersect $B(v_T,R_{i-1})$, except for the segments corresponding to $b^*$.
        Since those segments appear in exactly one circuit in the case of infinite $B_{\varepsilon ,\varepsilon '}$, and in a finite number of circuits in the case of finite $B_{\varepsilon ,\varepsilon '}$, the sum of circuits is thin, and we have successfully pushed $f$ to the tree.

    \end{proof}

    \begin{lemma}
        If there exists $R$ such that any circuit can be written as a $\ZZ_2$ \its\  of circuits of length $\leq R$, then $\cC(X,{\ZZ}_2) = \cC_R(X,{\ZZ}_2)$.
    \end{lemma}
    \begin{proof}
        Using compactness.
    \end{proof}

    \begin{corollary}
        For a finitely presented group $G$:
        \[
            \Huf_1(G,\ZZ_2) \cong  \Huf_1(T,\ZZ_2),
        \]
        which has dimension $\#\text{ends}(G) - 1$.
    \end{corollary}
    \begin{proof}
        By~\Cref{decomp_Z2}, the above lemma and the fact that any circuit in the Cayley graph of $G$ can be written as a finite sums of relators.
    \end{proof}

\subsection{Large circuits and dimensions}

    Our goal in this section is to show that, intuitively, if a graph has circuits of increasing length that can not be decomposed into smaller ones, then $\ZZ_2$-homology cannot vanish, and will in fact be infinite dimensional.
    
    Recall that a collection of elements of $\ZZ_2^{EX}$ is said to be \emph{thin} if it is locally finite.
    The sum of the collection (taken pointwise) is then well-defined, and we already defined:
    \begin{alignat*}{10}
        \cC   &:= \{\text{``thin sum of circuits''}\}, \\
        \cC_R &:= \{\text{``thin sum of circuits of length $\leq R$''}\}, \\
        \cC_\infty  &:= \cup_R \cC_R.
    \end{alignat*}

    The following are equivalent:
    \begin{enumerate}
        \item
            There exists an $R>0$ such that any circuit can be written as a \its\ sum of circuits of length $\leq R$.
        \item
            $\cC = \cC_R$.
        \item 
            For all $R'>R$, $\cC_{R'} = \cC_R$.
    \end{enumerate}
    If those do not hold, we say that $X$ has \emph{large circuits}.
    
    For the sake of clarity, we will settle on the following notation: 
    \begin{itemize}
        \item
            A \its\ is a well-defined sum of a potentially infinite thin collection $A \subseteq  \ZZ_2^{EX}$.
        \item
            There is an obvious notion of being closed under \its s, and a vector subspace of $\ZZ_2^{EX}$ is said to be a \its\ space if it is closed under \its s.
        \item
            The \its-span of a collection $A$ is the set:
            \[
                \{\sum B\ |\ B \text{ a thin subset of } A\}.
            \]
        By compactness, the \its\ span of a collection is closed under \its s (see~\cite{BG}).
        \item
           A collection $A$ is \its-independent if no two distinct thin subcollections of $A$ yield the same sum; equivalently, no thin subcollection has sum $0$.
        \item
            A \its\ basis is a collection that is \its-independent and a \its\ basis for a given \its\ space.
            Note that we do not require a \its\ basis $A$ for a \its\ space $\cW$ to be thin itself, but if it is, we get a bijection
            \[
                \ZZ_2^A \leftrightarrow  \cW. 
            \]
    \end{itemize}
    
    With those clarifications in mind, we see that each $\cC_R$ is a \its\ space, as is $\cC$ (as, respectively, the \its\ span of circuits of length $\leq R$, and all circuits), but $\cC_\infty $ is not necessarily.
    Finally, all the concepts above still make sense if we replace $\ZZ_2^{EX}$ by $\ZZ_2^{W}$ for any countable set $W$.

    We will show that as soon as $X$ has large circuits, then  $\Huf_1(X,{\ZZ}_2)$ is infinite-dimensional.

    To that end, we first construct a set of circuits in $X$ that is locally finite (thin) \emph{and} a \its\ basis for $\cC$.

    \begin{proposition}[$\cC$ has thin basis]\label{prop:cC_has_thin_basis}
        There exists a set $\cB$ of circuits that is a thin \its\ basis for $\cC$.
    \end{proposition}

    The proof uses an infinite variation of Gaussian elimination, which we will use again later in a transfinite way.

    Any element of $\cC$ can be uniquely written as a function $EX \rightarrow  {\ZZ}_2$.
    Order the edges $EX$ as $EX = (e_i)_{i\in {\NN}}$.
    For any non-zero element $f\in \cC$, there is then a well-defined \emph{leading index} $l(f)$ of $f$:
    \[
        l(f) := \min \{i\in {\NN}\ :\ f(i)\neq 0\}.
    \]
    If $f = 0$, we set $l(f)=\infty $ by convention, and will also set $g(\infty )=0$ for any $g\in \cC$.
    \begin{proof}
        Take all simple\footnote{simple $=$ no vertex appears twice in the circuit.} circuits $\cS$ and order them: $\cS = (f_j)_{j\in {\NN}}$ ($\cS$ is countable).
        We then apply Gaussian elimination on $\cS$ as follows.
        Define the following function inductively:
        \begin{alignat*}{10}
        F: \NN {\times}\NN  &\rightarrow  (\NN {\rightarrow}{\ZZ}_2)\\
        F_{0,j} &= f_j\\
        F_{t,j} &= \begin{cases}
        F_{j,j} &  \text{if $j<t$}\\                                            
        F_{t-1,j} + F_{t-1,j}(l(F_{t-1,t-1}))\cdot F_{t-1,t-1} &  \text{if $j\ge t$}\\
        \end{cases}
        \end{alignat*}
        (with index notation used for readability).
        Note that by our convention on leading index, the definition makes sense even for $F_{t-1,t-1}=0$.
        We first remark that for any $t>j$, $F(t,j) = F(j,j)$.
        Let then $g_j := F(j,j)$, and 
        \[
        \cB := \{g_j\ |\ j\in {\NN}, g_j \neq  0\}.
        \]
        Any element of $\cB$ is non-zero, and since for $j>i$, $g_j(l(g_i)) = 0$ by construction, it follows that they are all distinct.
        
        It remains to verify that: $\cB$ is thin, \its\ independent and a \its\ basis.
        
        \point{Thinness.}
        First, note that as soon as $i = l(g_j)$ for some $j$, then $e_i$ is in the support of finitely many elements of $\cB$ by construction.
        
        Assume now there exists an non-leading edge $e_i$ that is in the support of infinitely many elements of $\cB$, and without loss of generality that it is actually the first such (i.e. $i'<i$ implies $e_{i'}$ is in the support of finitely many circuits).
        Let $t$ be such that
        \[
            g_{t}(i) \neq  0 \text{ and } g_t(i') = 0\ \forall i'<i.
        \]
        The existence of such a $t$ is guaranteed by our assumption on $i$.
        Then $g_{t}(i)\neq 0$ but $g_{t}(i')=0$ whenever $i'<i$, so that $l(g_{t}) = i$, which is a contradiction.
        
        \point{Independence.}
        Given any non-empty set of elements $\{g_i\}_{i\in I}$ of $\cB$, consider $\hat i:=\min I$, and let $\hat j = l(g_{\hat i})$
        Then $(\sum_{i\in I} g_i)(\hat j) = 1$ since for all $I\ni i\neq \hat i$, $g_i(\hat j)=0$.
        Thus, no non-trivial sum is zero.

        \point{Basis.}
        Fix $h:\NN\cong EX\rightarrow\ZZ_2$ any element of $\cC$.
        By definition, $h$ is a thin sum of simple circuits $\sum_{i\in \mathcal{I}} f_i$, where $\mathcal{I}\subset \mathbb{N}$.
        By definition of $F$ we see that any $f_i$ can be written 
        \begin{alignat*}{10}
            f_i &= F_{i,i} + \sum_{t=1}^{i} F_{t-1,i}(l(F_{t-1,t-1}))\cdot F_{t-1,t-1}\\
        \intertext{and each summand is either zero or a basis element:}
                &= g_{i} + \sum_{t=1}^{i} F_{t-1,i}(l(F_{t-1,t-1}))\cdot g_{t-1}.
        \end{alignat*}
        Therefore, $f_i$ can also be seen as a map $\tilde{f}_i: \mathcal{B}\rightarrow\mathbb{Z}_2$, where each basis element is mapped to its coefficient in the sum above. 
        \renewcommand{\th}{\tilde{h}}
        Define now 
        \begin{alignat*}{10}
            \mathcal{I}_{s}:&=\{ i\in\mathcal{I}\ |\ i\le s \}\\
            h_s    :&= \sum_{i\in \mathcal{I}_s} f_i \in  (\ZZ_2)^{EX}\\
            \tilde{h}_s    :&= \sum_{i\in \mathcal{I}_s} \tilde{f}_i \in  (\ZZ_2)^\cB .
        \end{alignat*}
        Note that, by construction:
        \[
            h_s = \sum_g \tilde{h}_s(g)\cdot g.
        \]
        Note that if we fix an edge $e_r$, then the sequence $\{h_s(r)\}_{s\in\NN}$ is eventually constant (by thinness). 
        In other words, the sequence $\{h_s\}_{s\in \NN}$ converges (pointwise).
        Take an accumulation point and obtain the limit $\th$ of the sequence $\{\th_s\}_s$ and a subsequence $\{\th_{s_k}\}_{k}$ converging to $\th$.
        We now have 
        \[
            h_{s_k} \rightarrow  h, \quad \th_{s_k} \rightarrow  \th. 
        \]
        Define 
        \[
            h' := \sum_{g\in\mathcal{B}} \th(g)\cdot g.
        \]
        We see that for a fixed edge $e_r$,
        \begin{alignat*}{10}
        h'(r)
            &= \sum_{g\in\mathcal{B}} \tilde{h}(g)\cdot g(r) \\
            &= \sum_{g\in\mathcal{B}} \left(\lim_{k} \tilde{h}_{s_k}(g) \right)\cdot g(r) \\
            &= \sum_{g\in\mathcal{B}} \left(\lim_{k} \tilde{h}_{s_k}(g)\cdot g(r) \right) \\
        \intertext{and since $g(r)=0$ for all but a finitely number of $g\in \cB$, we can exchange limit and sum:}
            &= \lim\limits_{k} \sum_{g\in\mathcal{B}} \tilde{h}_{s_k}(g)\cdot g(r)\\
            &= \lim\limits_{k} h_{s_k}(r)\\
            &= h(r).
        \end{alignat*}
        This shows that any $h\in \cC$ can be written as a \its\ of elements of $\cB$, and concludes the proof.	
    \end{proof}

    Let $\cB$ denote a thin \its\ basis consisting of circuits as constructed above.
    By thinness, any set of elements of $\cB$ defines a (legal, i.e. thin) sum, and thus, we get a bijection:
    \[
    \ZZ_2^\cB  \leftrightarrow  \cC.
    \]

    We now want to tweak $\cB$ so as to have a sequence of increasing thin bases $\cB_i'$ for $\cC_i$ and $\cB'$ for $\cC$ such that:
    \[
    \cB_i' \subseteq  \cB_{i+1}' \subseteq  \cB'.
    \]
    This can be done in a somewhat abstract setting:

    \begin{proposition}\label{proposition:filtered_basis}
        Assume a \its\ space $W \subseteq  \ZZ_2^X$ is given, along with a sequence of increasing \its\ subspaces $V_i$.
        If $\cW$ is a countable thin \its\ basis for $W$, then one can construct a countable, thin, \its\ basis $\cB$ for $W$, along with sequence of nested \its\ bases $\cB_i$ for $V_i$.
    \end{proposition}

    Since $\cW$ is thin and a basis, we have $W\leftrightarrow \ZZ_2^\cW $.
    Given any well-ordering of $\cW$, one can let, for any $v\in W$:
    \[
    l(v) = \min\{i: v(i)\neq 0 \text{ when $v$ is viewed as an element of $\ZZ_2^\cW $}\}.
    \]
    If $v=0$, we use the same convention as above: $l(v)=\infty $, and for any $w$, $w(\infty )=0$.
    In the proof, we will assume a fixed ordering on $\cW$ and identify $W\leftrightarrow \ZZ_2^\cW $. 

    \begin{proof}

        Well-order the set $W$ in a way that respects the subsets $V_i$, i.e.
        \[
            W = \{v_\alpha \ |\ \alpha <\kappa \},
        \]
        with $\kappa $ the cardinal of $W$ and such that there exist increasing ordinals ${\alpha}_i$ with
        \[
            V_i =\{v_\alpha \in W\ |\ \alpha <{\alpha}_{i+1}\}.
        \]
        We show that we can construct $\cB_i$ ($i\in \NN$) and $\cB$ satisfying these requirements.
        We define the “Gaussian elimination” map 
        \[
        \nu :\kappa {\times}\kappa {\times}\NN\rightarrow  \ZZ_2
        \]
        inductively as follows:
        \begin{alignat}{10}\label{gaussian_elim_def}
            {\nu}_{0,\alpha } &=   v_\alpha ,\nonumber\\
            {\nu}_{\tau ,\alpha } &= \begin{cases}
            \lim_{\sigma <\tau } {\nu}_{\sigma ,\alpha } & \text{if $\tau $ is a limit ordinal},\\
        {\nu}_{\alpha ,\alpha }    & \text{otherwise and $\tau >\alpha $},\\                                         
        {\nu}_{\tau -1,\alpha } - {\nu}_{\tau -1,\alpha }(l({\nu}_{\tau -1,\tau -1}))\cdot {\nu}_{\tau -1,\tau -1} &  \text{otherwise and $\tau \le \alpha $},\\
        \end{cases}
        \end{alignat}
        where the limit is a pointwise limit (which means that at any coordinate $r$, the transfinite sequence, evaluated at that coordinate, is eventually constant).
        
        From now on, we will write ${\nu}_\alpha $ for the diagonal element ${\nu}_{\alpha ,\alpha }$.
        
        The interpretation of the map $\nu $ is as follows.
        An element of $W$ is uniquely determined by a map $\NN\rightarrow \ZZ_2$ (a row).
        Thus, the set $W$, once ordered, can be represented as a map $\kappa {\times}\NN\rightarrow \ZZ_2$ (a matrix).
        Gaussian elimination, in the finite case, is an iterative process where one iterates over the rows of a matrix, each the “pivot”, while modifying the matrix at each step, hence a map $\kappa {\times}\kappa {\times}\NN\rightarrow \ZZ_2$ .
        In short, the first coordinate represents the iteration/pivot, the second the element in question, and the third its value at a given base coordinate.
        
        At time zero, the matrix has not been modified.
        At time $\tau $, if $\tau $ is a successor ordinal, we pivot around the element ${\nu}_{\tau -1,\tau -1}$. Note that if the row $\alpha $ ``lies before the pivot” or ``is the pivot '' (i.e. $\alpha <\tau $), it is not changed.
        If row $\alpha $ “lies after the pivot” ($\alpha \ge \tau $), we must subtract the pivot to what row $\alpha $ “just was”. 
        In case $\tau $ is a limit ordinal, then we will just re-iterate all changes that happened before time $\tau $ and restart pivoting in its successor.

        We start by showing by induction on $\tau $ that:
        \begin{enumerate}
            \item\label{prop:thin}
                For any $n\in \NN$, the number of ordinals $\sigma {\leq}\tau $ such that ${\nu}_{\sigma }(n)\neq 0$ is finite.
            \item\label{prop:lim_well_def}
                The limit in \Cref{gaussian_elim_def} is actually well-defined.
            \item\label{prop:zero_under_leading_index}
                For any $\sigma <\tau ,\alpha $, ${\nu}_{\tau ,\alpha }(l({\nu}_{\sigma })) = 0$.
            \item\label{prop:defining_sum}
                The value of ${\nu}_{\tau ,\alpha }$ is
                \begin{equation}\label{transfiniteind}
                    {\nu}_{\tau ,\alpha } = {\nu}_{0,\alpha } + \sum_{\sigma <\min\{\alpha ,\tau \}} {\nu}_{ \sigma ,\alpha }(l({\nu}_{\sigma })) \cdot {\nu}_{\sigma }.
                \end{equation} 
        \end{enumerate}

        Indeed, assume the above holds for any $\theta <\tau $.

        \point{If $\tau =0$:} 
        Then \Cref{prop:thin,prop:lim_well_def,prop:zero_under_leading_index,prop:defining_sum} hold trivially.
        \point{If $\tau $ is successor:}
        
        ~\Cref{prop:thin} is easy. By the induction hypothesis the number of  $\sigma \le \tau -1$ with ${\nu}_{\sigma }(n)\neq 0$ is finite. Hence it must be finite for $\tau $ too.
        
        ~\Cref{prop:lim_well_def} does not apply successor ordinals.
        
        For~\Cref{prop:zero_under_leading_index}, we see that if $\tau >\alpha $, then
        \[
        {\nu}_{\tau ,\alpha }(l({\nu}_{\sigma })) = {\nu}_{\alpha }(l({\nu}_{\sigma }))
        \]
        which is zero by induction hypothesis (i.e.~\Cref{prop:zero_under_leading_index} holds at $\theta =\alpha $).
        If $\tau {\leq}\alpha $, then
        \begin{equation}
        {\nu}_{\tau ,\alpha } = {\nu}_{\tau -1,\alpha } - {\nu}_{\tau -1,\alpha }(l({\nu}_{\tau -1,\tau -1})) \cdot {\nu}_{\tau -1,\tau -1},
        \end{equation} 
        which, evaluated in $l({\nu}_{\tau -1,\tau -1})$ is zero by construction.
        If $\sigma <\tau -1$, both ${\nu}_{\tau -1,\alpha }(l({\nu}_{\sigma }))$ and ${\nu}_{\tau -1,\tau -1}(l({\nu}_{\sigma }))$ are zero, hence so is ${\nu}_{\tau ,\alpha }$.

        Finally, ~\Cref{prop:defining_sum}. One sees that if $\tau {\leq}\alpha $, we get $\min\{\tau ,\alpha \} = \tau $, and:
        \begin{alignat*}{10}
            {\nu}_{\tau ,\alpha } 
                &= {\nu}_{\tau -1,\alpha } + {\nu}_{ \tau -1,\alpha }(l({\nu}_{\tau -1,\tau -1})) \cdot {\nu}_{\tau -1,\tau -1} \\
                &= {\nu}_{0,\alpha } + \sum_{\sigma <\tau -1} {\nu}_{ \sigma ,\alpha }(l({\nu}_{\sigma })) \cdot {\nu}_{\sigma } + {\nu}_{ \tau -1,\alpha }(l({\nu}_{\tau -1,\tau -1})) \cdot {\nu}_{\tau -1,\tau -1},\\
        \intertext{by definition of ${\nu}_{\tau ,\alpha }$. If $\tau >\alpha $, we get $\min\{\tau ,\alpha \} = \alpha $, and}
        {\nu}_{\tau ,\alpha } = {\nu}_{\tau -1,\alpha } &= {\nu}_{0,\alpha } + \sum_{\sigma <\min\{\alpha ,\tau -1\}} {\nu}_{\sigma ,\alpha }(l({\nu}_{\sigma })) \cdot {\nu}_{\sigma }\\
                &=  {\nu}_{0,\alpha } + \sum_{\sigma <\min\{\alpha ,\tau \}} {\nu}_{\sigma ,\alpha }(l({\nu}_{\sigma })) \cdot {\nu}_{\sigma }
        \end{alignat*}
        
        Now moving on to the limit case.

        \point{If $\tau $ is limit:}
        We first check~\Cref{prop:thin}.
        It suffices to show that for any $n$, the number of ordinals $\sigma <\tau $ with ${\nu}_{\sigma }(n)\neq 0$ is finite, since then the number of $\sigma {\leq}\tau $  with ${\nu}_{\sigma }(n)\neq 0$ will also be finite.
        If $n = l({\nu}_{\sigma })$ for some $\sigma <\tau $, then the number of $\theta <\sigma $ for which ${\nu}_{\theta }(n)=0$ is finite by the induction hypothesis (\Cref{prop:thin}), and for any $\tau >\theta >\sigma $, ${\nu}_{\theta }(n)=0$, also by the induction hypothesis (\Cref{prop:zero_under_leading_index}).
        Assume then that there exists some $n$ which is not a leading coefficient, which is also in the support of infinitely many ${\nu}_\sigma $ ($\sigma <\tau $).
        We can take $n$ to be the least such, and for $\sigma $ large enough, we will have:
        \[
            {\nu}_\sigma (n) \neq  0, \quad {\nu}_{\sigma }(k) = 0,\ \forall k<n,
        \]
        so that $n$ is actually the leading coefficient of ${\nu}_\sigma $, a contradiction.

        We now check~\Cref{prop:lim_well_def}.
        Let $\alpha $ arbitrary, then
        \[
            {\nu}_{\tau ,\alpha } = \lim_{\sigma <\tau } { \left({\nu}_{0,\alpha } + \sum_{\sigma '<\min\{\alpha ,\sigma \}} {\nu}_{ \sigma ',\alpha }(l({\nu}_{\sigma '})) \cdot {\nu}_{\sigma '} \right)}
        \]
        which we need to verify to be well-defined.
        In other words, we need to check that for any $n$, the sequence
        \[
            \sigma \  \mapsto \  {\nu}_{0,\alpha }(n) + \sum_{\sigma '<\min\{\alpha ,\sigma \}} {\nu}_{ \sigma ',\alpha }(l({\nu}_{\sigma '})) \cdot {\nu}_{\sigma '}(n)
        \]
        stabilises starting at some ordinal.
        But by~\Cref{prop:thin}, ${\nu}_{\sigma '}(n)$ is zero for $\sigma '$ large enough, so that the sequence indeed stabilises.
        Furthermore, it stabilises to:
        \[
           {\nu}_{0,\alpha }(n) + \sum_{\sigma <\min\{\alpha ,\tau \}} {\nu}_{ \sigma ,\alpha }(l({\nu}_{\sigma })) \cdot {\nu}_{\sigma }(n) ,
        \]
        which also proves~\Cref{prop:defining_sum}.
        
        It remains to check~\Cref{prop:zero_under_leading_index}.
        But if $\sigma <\tau ,\alpha $, then
        \[
            {\nu}_{\tau ,\alpha }(l({\nu}_\sigma )) = \lim_{\sigma '<\tau } {\nu}_{\sigma ',\alpha }(l({\nu}_\sigma )) = \lim_{\sigma <\sigma '<\tau }  {\nu}_{\sigma ',\alpha }(l({\nu}_\sigma )) = \lim 0 = 0.
        \]
        
        This closes the induction.

%
        Let now
        \[
            \cB := \{{\nu}_\alpha  | \alpha  <\kappa , {\nu}_\alpha {\neq}0\},\quad \cB_i = \{{\nu}_\alpha  | \alpha  <{\alpha}_{i+1}, {\nu}_\alpha {\neq}0\}
        \]
        Note that by~\Cref{prop:thin}, each non-zero ${\nu}_\alpha $ has a well-defined leading index $l({\nu}_\alpha )$, and ${\nu}_\alpha {\neq}{\nu}_\beta $ implies $l({\nu}_\alpha )\neq l({\nu}_\beta )$. 
        It follows that $\cB$ is at most countable.
        We can now check the following consequences of this process:
        \begin{enumerate}
            \item 
                By~\Cref{transfiniteind}, $\cB_i$ generates $V_i$ and $\cB$ generates $W$.	 
                Indeed,~\Cref{transfiniteind} yields (by setting $\alpha =\tau $)
                \[
                    {\nu}_{\tau } = {\nu}_{0,\tau } + \sum_{\sigma <\tau } {\nu}_{\sigma ,\tau }(l({\nu}_{\sigma })) \cdot {\nu}_{\sigma },
                \]
                so that, by rewriting the sum, we get
                \[
                    {\nu}_{0,\tau } = {\nu}_{\tau } - \sum_{\sigma <\tau } {\nu}_{\sigma ,\tau }(l({\nu}_{\sigma })) \cdot {\nu}_{\sigma },
                \]
                so that any element of $W$ (resp. $V_i$) is a thin sum of elements of $\cB$ (resp. $\cB_i$).
            \item 
                $\cB$ is thin, as a subset of  $\ZZ_2^\cW $, by the same argument as the proof of ~\Cref{prop:thin} for $\tau $ limit. 
            \item 
                $\cB$ is thin, as a subset of $\ZZ_2^X$. 
                Every $x\in X$ appears in the support of a finite number of $v\in \cW$, and every $v\in \cW$ in the support of a finite number of elements of $\cB$.
            \item 
                The set $\cB$ is \its\ independent. 
                Suppose not. 
                Then, for some $\tilde \kappa \subset \kappa $, we have $\sum_{\sigma \in\tilde \kappa } {\nu}_\sigma  = 0$. 
                Let ${\tau}_0:=\min \{\sigma  \in \tilde \kappa  \ |\ {\nu}_\sigma  \ne 0  \}$.
                Then $\sum_{\sigma \in\tilde \kappa } {\nu}_\sigma  (l({\nu}_{{\tau}_0})) =  1$, a contradiction.
        \end{enumerate}
    \end{proof}

    In our case, this yields:
    \begin{corollary}\label{cor:filtered_basis_for_cycle_spaces}
        There exists a thin \its\ basis $\cB$ for $\cC$ with nested subsets $\cB_i$, each a \its\ basis for $\cC_i$.
    \end{corollary}
    \begin{proof}
        By~\Cref{prop:cC_has_thin_basis} and~\Cref{proposition:filtered_basis}. 
    \end{proof}

    And we can now easily get:

    \begin{theorem}\label{thm_large_circuits_imply_inf_dim}
        If $X$ has large circuits, then $\Huf_1(X,{\ZZ}_2)$ is infinite dimensional (as a vector space). 
    \end{theorem}
    \begin{proof}
        Let $\cB,\cB_i$ be as in~\Cref{cor:filtered_basis_for_cycle_spaces}.

        If $X$ has large circuits, the sequence
        \[
            \cC_1 \subseteq   \ldots \subseteq  \cC_i \subseteq  \cC_{i+1} \subseteq  \ldots
        \]
        does not stabilise, so that there exists a sequence $(i_j)_{j\in {\NN}}$ satisfying:
        \[
            \cC_{i_j} \subsetneq  \cC_{i_{j+1}}.
        \]
        In particular, this implies that $\cB_{i_j} \subsetneq  \cB_{i_{j+1}}$.
        Choose then, for each $j$, $b_j \in  \cB_{i_{j+1}} - \cB_{i_{j}}$.
        
        Let $\cJ$ be an infinite set of infinite pairwise disjoint subsets of $\NN $ and
        \[
            f_J := \sum_{j\in J} b_j, \quad J \in  \cJ.
        \]
        We claim that the family $\{[f_J] \in  \Huf_1(X,\ZZ_2)\}_{J\in \cJ}$ is linearly independent, which will imply infinite dimensionality.

        We can actually show more:
        the family $\{f_J\}_{J\in \cJ}$ is thin, and no \its\ sum of the form:
        \[
            \sum_{J\in \cJ'} f_J,\quad \cJ' \subseteq  \cJ,
        \]
        lies in $\cC_\infty $.
        Indeed, since the $f_J$s have disjoint support, any such combination can be written as an infinite sum:
        \[
            f = \sum_{j\in J'} b_j,
        \]
        for some $J'\subseteq  \NN$.
        If $f \in  \cC_\infty $, then there would exist $R$ with $f\in \cC_R$, so that $f = \sum_{b\in B}b$ for $B \subseteq  \cB_R$, and $0 = \sum_{B_{J'}\Delta B} b$, where $B_{j'} = \{b_j\ |\ j \in  J'\}$.
        Since $J'$ is infinite, $B_{J'}$ contains elements of $\cB_{R'}$ for $R'$ arbitrarily high, so that the symmetric difference above is not empty.
        This is a contradiction with $\cB$ being \its\ independent.
        
    \end{proof}

    In $\Zsuf_1(X,\ZZ_2)$, infinite sums are allowed, as long as they are locally finite.
    However, in $\Huf_1(X,\ZZ_2)$, only \emph{finite} sums are naturally permitted; there is no obvious ``\its" space structure on quotients.
    Thus, the passage to a quotient involves a major loss of information.  We suggest a very simple approach that remedies this problem in our ad hoc situation.  It makes sense to define a thin quotient of thin vector spaces. There might be a more abstract setting about which one could say something.
	
	\begin{definition}
        Fix a \its~space $V\leq \ZZ_2^X$, and a sub-vector space $W\leq V$ (not necessarily closed under \its). 
        In the quotient $V/W$, we say that a potentially infinite sum of elements $\sum a_i$ is \emph{consistent} iff there is a collection of preimages $\tilde a_i \in  a_i$ such that the collection $\{\tilde a_i\}_i$ is thin (in particular, abusing notation, any $v\in V-\{0\}$ appears as at most a finite number of $\tilde a_is$), and all such thin choices of representatives result in the same element in $V/W$:
        \[
            \sum  \tilde a_i - \sum  \tilde b_i \in  W,
        \]
        if $\{\tilde a_i\}_i, \{\tilde b_i\}_i$ are two choices of representatives.
	    
    \end{definition}

	In order to have a cleaner theorem we also introduce the following “thin” sum.
	

%
%
%
%
%
    \begin{theorem}\label{predimensiontheorem}
        The following are equivalent:
        \begin{enumerate}
            \item In $\Huf_1(X,\mathbb{Z}_2)$, the infinite sum $\{[0]\}_{i=0}^\infty $ of the zero-class is consistent.
            \item $X$ has no large circuits.
            \item Any sum of cosets that has a thin representative sum is consistent.
        \end{enumerate}
    \end{theorem}
    \begin{proof}
        Clearly 3 implies 1.\newline
        
        We show 1 implies 2. Suppose $X$ has large circuits, then there exists an infinite thin sum of large circuits that is non-zero in homology. However every individual circuit is finite and thus zero in homology. On the other hand the repeated infinite sum of $0$ gives $0$. So we see that the infinite sum of the zero class has two representing thin sums that give a different solution. Hence it is inconsistent.\newline
        
        We show 2 implies 1. Consider an infinite representing thin sum $\sum\limits_{i=1}^\infty  \tilde x_i$ of the infinite zero class sum. Since there are no large circuits, there is some $R>0$ such that each $\tilde x_i$ can be written as a sum of $R$-circuits. More specifically there is a $g_i$ such that $\partial g_i =  \tilde x_i$. Now take some accumulation point $\lambda$ of the sequence $\left\{ \sum\limits_{i=1}^{n} g_i \right\}_{n\in\mathbb{N}}$. Then $\partial \lambda = \sum\limits_{i=1}^\infty  \tilde x_i $. The infinite zero-class sum is as such consistent.\newline
        
        We show 1 implies 3. Suppose there is some inconsistent thin sum. So two representing sums $\sum\limits_{i=1}^\infty \tilde x_i$ and  $\sum\limits_{i=1}^\infty \tilde y_i$, such that $[\tilde x_i] = [\tilde y_i]$. Then each $[\tilde x_i-\tilde y_i]=[0] $. Hence by assumption $\left [\sum\limits_{i=1}^\infty \tilde x_i-\tilde y_i\right ]=[0] $. Which shows that $\left [\sum\limits_{i=1}^\infty \tilde x_i\right]=\left [\sum\limits_{i=1}^\infty \tilde y_i\right] $, because finite sums are always well defined.
    \end{proof}

    \begin{corollary}\label{coroldimension}
        If $\Huf_1(X,\mathbb{Z}_2)$ has large circuits, then every \its~is inconsistent.
    \end{corollary}
    \begin{proof}
        Direct consequence of Theorem \ref{predimensiontheorem}.
    \end{proof}

    We draw the following conclusion:

    Define the \emph{thin dimension} $\dim_{\text{thin}}$ of $\Huf_1(X,\ZZ_2)$ as the least cardinal $\kappa $ such that $\Huf_1(X,\ZZ_2)$ has a subset $K$ of cardinality $\kappa $ and such that any element of $\Huf_1(X,\ZZ_2)$ can be written as a consistent sum of elements of $K$.

    \begin{theorem}
        \begin{itemize}
            \item $\dim_{\text{thin}} (\Huf_1(X,\mathbb{Z}_2))$ is countable if and only if $X$ has no large circuits. In this case it is equal to the number of ends $-1$ (where we do not make the difference between countably many or uncountably many ends).
            \item  $\dim_{\text{thin}} (\Huf_1(X,\mathbb{Z}_2))$ is uncountable if and only if $X$ has large circuits.
        \end{itemize}
    \end{theorem}
    \begin{proof}
        
        \point{First part.} If there are no large circuits, then by~\Cref{decomp_Z2} there exists a tree $T$ inside $X$, such that $ \Huf_1(T,\ZZ_2)=\Huf_1(X,\ZZ_2)$. The isomorphism is given by the natural inclusion map. Hence~\Cref{decomp_Z2} gives us a thin \its\ basis $\cB$ in the tree. Note that all sums remain well-defined when we inject $\cB$ into $X$ by Theorem \ref{predimensiontheorem}.\newline
        
        \point{Second part.} Suppose $X$ has large circuits, then by Corollary \ref{coroldimension} all consistent sums are finite. However we showed before that $\Huf_1(X,\mathbb{Z}_2)$ contains an uncountable number of elements. So a generating set must be uncountable too.\newline
        
        Note that since having large circuits and not having large circuits exhausts the universe, both implications above are actually equivalences.
    \end{proof}

%% file: ZZ_Large_Circuits.tex

\section{$\ZZ$-large circuits}\label{section:ZZ_large_circuits}

Recall that a ULF graph $X$ has $\ZZ$-large circuits if
\[
    \forall r>0\ \cC_r(X,\ZZ) \neq  \cC_\infty (X,\ZZ).
\]

\begin{proposition}\label{ZZ_large_circuits_implies_H_neq_zero}
    If $X$ is vertex-transitive and has $\ZZ$-large circuits, then $\Huf_1(X,\ZZ)\neq 0$.
\end{proposition}
\begin{proof}
    
    Assume $X$ vertex transitive with zero homology and $\ZZ$-large circuits.
    For any $r>0$, take some $f_r \in \cC_\infty  - \cC_r$.

    We claim that $f_r$ can be assumed to have norm $1$.
    First, by assumption that $f \in  \cC_\infty $, there exists some $s>r$ with $f_r \in  \cC_s$, and thus $f_r = \fe g$ for some $g:\fC_s\rightarrow \ZZ$ uniformly bounded.
    One can write $g = g^{+} - g^{-}$, so just assume for now that $g\geq 0$.
    Now, consider the following process: view $g$ as a multiset, and consider a maximal subset of circuits that do not intersect pairwise, and subtract it from $g$; iterate.
    This process has to end eventually (indeed, fix any circuit $c$ in the support of $g$; by ULF, there exists some $K$ not depending on $c$ such that $c$ intersects at most $K$ other circuits in the support of $g$; if during the $K$ first steps $c$ is never removed, this means by maximality that at each step a neighbor is removed; thus after $K+1$ steps all neighbors are removed, and then so will $c$, by maximality), so that $f = \sum  \fe g^{+}_i - \sum  \fe g^{-}_i$, with each $f_i^\pm  := \fe g^{\pm}_i$ of norm $1$.
    If each $f_i^\pm $ lied in $\cC_r$, so would $f$ ($\cC_r$ being a vector space).
    We conclude that at least one of $f_i^\pm $ does not lie in $\cC_r$, and the claim is verified.
    
    Fix some basepoint $v_0\in VX$, and for any $r$, some $f_r = \fe g_r$ as above, and consider the family $(f_r^{i} \in  \Csuf_1(X,\ZZ))_{i\in \NN}$ constructed as follows:
    For any $i$, take the restriction of $g_r$ to circuits intersecting a ball of radius $i$ around $v_0$, call this $g_r^i$ and let  $f_r^i = \fe g_r^i$.
    Note that since all circuits of $g_r^i$ are disjoint, $f_r^i$ also has norm $1$; furthermore, $f_r^i$ has support in a ball of radius $i+s$ if $g_r \in  l^\infty (\fC_s)$.
    Im other words, $f_r$ and $f_r^i$ are both elements of $\Zsuf_1$ of norm ($\leq $) $1$ and agreeing on the ball of radius $i$ around $v_0$.
    The important aspect of this family is that $f_r^i \ovs{i}{\rightarrow } f_r$ pointwise.

    Since $X$ is infinite and transitive, it is now possible to move the elements $f_r^i$ and balls $B_{r,i}$ containing their support, for $i$ \emph{and} $r$ varying in $X$ in such a way that the balls do not intersect pairwise.
    Say ${\phi}_{r,i}$ is the automorphism of $X$ used to move $f_r^i$.

    Let $F$ be the well-defined sum of these (moved elements); $F$ lies in $\Zsuf_1$ and has norm $1$.
    Since $\Huf_1=0$, there exists some $F \in  \cC_\infty $, and by definition, $F \in  \cC_s$ for some $s$.
    Take therefore some $G:\fC_s \rightarrow  \ZZ$ with $\fe G=F$ and of uniformly bounded norm.
    Fix any $r>s$ and consider the sequence $(G_{r,i}) := (G\circ {\phi}_{r,i}^{-1})_i$.
    Note that by construction $\fe G_{r,i}|B_{r,i} = f_{r,i}$ and recall that $f_{r,i} \rightarrow  f_r$.
    Up to taking a subsequence, one may assume that $G_{r,i}$ converges to some $G_r$ (since the $G_{r,i}$ are uniformly bounded in norm).
    Therefore, $G_{r,i} \rightarrow  G_r$, which implies $\fe G_{r,i} \rightarrow  \fe G_r$, but $\fe G_{r,i} \rightarrow  f_r$ since $\fe G_{r,i}|B_{r,i} = f_r$ and the sequence $B_{r,i}$ exhausts $X$.
    Since $\fe G_r=f_r$, we get $f_r\in \cC_s$, thus contradicting $f_r \notin  \cC_r$.

\end{proof}

\begin{proposition}
	If $X$ is vertex-transitive and has $\ZZ$-large circuits, then $\dim(\Huf_1(X,\ZZ))=\infty$.
\end{proposition}
\begin{proof}
	Consider the function $f$ constructed in the proof of Proposition \ref{ZZ_large_circuits_implies_H_neq_zero}. In the same way it is possible to construct instead countably many functions $f_k$ each defined on balls $B^k_{r,i}$, which are the building blocks of the construction. Moreover these $f_k$ can be chosen to not have any intersecting balls between them. By arguments of the proof before they are all non-zero in homology and finite linear combinations between them (as well as bounded infinite linear combinations for that matter) are also non-zero. Hence the result follows. 
\end{proof}

%% file: Expansion.tex

\section{Expansion}\label{section:expansion}

In this section, we follow a different trail in our search for a description of vanishing $\Huf_1(\cdot ,\ZZ)$.
Our motivation is twofold:
On the one hand, after playing a bit with $\Huf_1(\cdot ,\ZZ)$, one comes to the conclusion that, ends and “large circuits” aside, the main ingredient in non-vanishing of homology is a lack of triangles, compared to edges.
The typical example here is (a triangulated) $\ZZ^2$: Any infinite sum of parallel bips going in one direction will define a non-zero class in homology, since there are not enough triangles to kill all bips at the same time.
On the other hand, (classical) graph expansion can be interpreted as the ability, given any finite subset of vertices $U$, and element $f \in  \Zsuf_0(U)$, to push $f$ out of $U$ in a bounded way; that is, there exists some $g\in \Csuf_1(X)$ with $\partial g|U = f$ and $\|g\|\leq  {\varepsilon}^{-1}\|f\|$, where $\varepsilon $ is the Cheeger constant for $X$.
This section can then be described as an effort to try and mimic the “homological interpretation” of expansion in higher dimension (we call this $\Hsuf_n$-expansion), and fine-tune it to isolate the phenomenon of “lack of triangles” from “ends” and “large circuits” (thus getting \emph{pure} $\Hsuf_n$-expansion).

Recall that a (ULF) graph $X$ is said to be \emph{amenable} if it satisfies the following condition:
\[
    \forall  \epsilon >0 \ \exists  U {\subseteq}_f VX\ :\ |\partial U| < \varepsilon |U|.  
\]
If $X$ is non-amenable, the least $\varepsilon >0$ satisfying $\forall  U {\subseteq}_f VX$, $ |\partial U|\geq {\varepsilon}|U|$ is called the Cheeger constant for $X$.

The following was proved in~\cite{BW}:
\begin{proposition}[Part of {\cite[Theorem 3.1]{BW}}]\label{thm:amenability}
    A ULF graph $X$ is non-amenable iff $\Huf_0(X,\ZZ ) = 0$. 
\end{proposition}

\subsection{Basics}

Let $X$ be a ULF simplicial complex.
If $U\subseteq VX$ is a set of vertices, we also use $U$ for the induced simplicial subcomplex.
We then write $\partial U$ for the simplices of $X$ adjacent to a vertex not in $U$.

\begin{definition}[$\Hsuf$-Expansion]
    We say that $X$ has $n$-dimensional $\Hsuf$-expansion if:
    There exists a function $K:\NN {\rightarrow}\NN $ satisfying the following.
    Fix any finite set of vertices $U{\subseteq}_f VX$, and any $f \in  \Csuf_n(U,\ZZ )$ satisfying:
    \begin{enumerate}
        \item 
            $\partial f$ has support contained in $\partial U$.
        \item\label{extensible}
            $f$ can be extended to some $\hat f \in  \Zsuf_n(X,\ZZ )$ with $\|\hat f\| =  \|f\|$ (by extended, we mean $\hat f |U = f$).
    \end{enumerate}
    Then: There exists $g\in \Csuf_{n+1}(X,\ZZ )$ such that $\partial g|_U = f$ and $\|g\|\leq K(\|f\|)$.

    We say that a a function $K$ as above is a certificate to $\Hsuf_n$-expansion.

    Finally, if~\Cref{extensible} in the definition is replaced by
    \begin{enumerate}
        \item[3.]
            $f$ can be extended to some $\hat f\in  \Bsuf_n(X,\ZZ)$ with $\|\hat f\| =  \|f\|$ .
    \end{enumerate}
    and the rest is leaved as-is, then the obtained condition is called \emph{pure} $\Hsuf_n$-expansion.

\end{definition}

Intuitively, expansion can be interpreted as homological cycles locally being homological boundaries, with control.
In other words, one can always kill “pieces of cycles” with “pieces of boundaries” in a bounded way.
We will see that this property \emph{always} implies that simplicial uniformly finite homology vanishes, at the corresponding dimension.
Pure expansion is slightly harder to describe, but the core idea is that it only takes into account those $f$ that “can actually be killed” by some boundary.

Note that 
\[
    \text{$\Hsuf_n$-expansion} \quad \Rightarrow  \quad \text{pure $\Hsuf_n$-expansion}.
\]

\begin{lemma}\label{expansion_implies_Huf_vanishes}
    If $X$ has $\Hsuf_n$-expansion, then $\Hsuf_n(X,\ZZ ) = 0$.
\end{lemma}
\begin{proof}
    Assume that $K$ is a certificate.
    Fix $f \in  \Zsuf_n(X,\ZZ )$, and $\langle U_n \subseteq  VX{\rangle}_n$ a nested sequence of finite subsets exhausting $VX$.
    For each $n$, let $f_n$ the restriction of $f$ to $U_n$, and assume that $\|f_n\| = \|f\|$ for each $n$ (achievable by starting with $U_0$ large enough).
    Since $\|f\| = \|f_n\|$, let $g_n \in  \Cuf_{n+1}(X,\ZZ )$ satisfying $\partial g_n|U_n = f_n$ and $\|g_n\| \leq  K(\|f\|)$, as given by $\Hsuf_n$-expansion.

    Then the sequence $\langle g_n{\rangle}_n$ has an accumulation point, say $g$, since its elements are uniformly bounded (by $K(\|f\|)$).
    Since $\partial g$ is also an accumulation point of $(\partial g_n)_n$, which converges to $f$, we conclude that $\partial g = f$.

    Thus, $f \in  \Buf_n(X,\ZZ )$ and $\Hsuf_n(X,\ZZ ) = 0$.
\end{proof}

\begin{remark}
    Consider the following condition on $X$:
    There exists some constant $N$ such that for any $f\in \Zsuf_n(X,\ZZ)$, one can write $f = \sum_{i=1}^{N\|f\|}f_i$, with each $f_i$ of norm $1$ and in $\Zsuf_n(X,\ZZ)$.
    If this condition held, we could always take the function $K:\NN\rightarrow \NN$ in $\Hsuf_n$-expansion to be linear.
    We do not know whether it holds, even in dimension $1$ (that is: can one decompose any flow of norm $k$ into a sum of $Nk$ flows of norm $1$, for some universal $N$ ?)
\end{remark}

In the zero-dimensional case, $\Hsuf$-expansion is actually equivalent to non-amenability. Moreover since $B_0^{suf} = Z_0^{suf}$, also pure $\Hsuf_0$ expansion coincides with $\Hsuf_0$ expansion.
\begin{proposition}[$\Hsuf_0$-expansion is just expansion.]
    Let $X$ be a ULF graph.
    Then $X$ has $\Hsuf_0$-expansion iff $X$ is non-amenable.
\end{proposition}
\begin{proof}
    By~\Cref{thm:amenability} and~\Cref{expansion_implies_Huf_vanishes}, we get $\Hsuf_0$-expansion $\Rightarrow $ non-amenability.

    Now for the converse:
    By decomposing such an $f:U\rightarrow \ZZ$ into a sum of $0/1$-valued function, the condition reduces to:
    \begin{itemize}
    \item
        There exists $K\in {\NN}$ such that for all $U{\subseteq}_f VX$ and $W\subseteq U$, there exists $g:EX\rightarrow \ZZ$ with $\partial g|U = {\chi}_W$ and $\|g\|\leq K$.
    \end{itemize}

    Assume now that $X$ is non-amenable, with Cheeger constant $\varepsilon $.
    By adapting the proof of \cite[Lemma 2.1]{BS} (more precisely, their construction of a flow), we see that given $U {\subseteq}_f VX$ and $W\subseteq  U$, one can construct a $g$ satisfying $\partial g|U = {\chi}_W$.
    Since in their argument, the “flux” (=excess flow) at vertices in $W$ is $\varepsilon $ and the flow at edges at most $1$, scaling by ${\varepsilon}^{-1}$ yields the bound of $K:= {\varepsilon}^{-1}$ for $g$.

\end{proof}

\subsection{Dimension $1$}

\begin{theorem}\label{transitive_graphs_expansion}
    If $X$ is an infinite, vertex transitive ULF simplicial complex and $\Hsuf_1(X,\ZZ ) = 0$, then $X$ has $\Hsuf_1$-expansion. 
\end{theorem}
\begin{proof}
    Suppose $X$ does not have $\Hsuf_1$-expansion but $\Hsuf_1(X,\ZZ) = 0$.
    Then, there exist sequences $\langle U_i {\subseteq}_f VX\rangle $ and $\langle f_i \in  \Csuf_1(X,\ZZ){\rangle}_i$ as in the definition, with $\|f_i\|\leq C$ for some $C$, but such that if $\partial g_i|U_i = f_i$, then $\|g_i\|\geq K_i$, with $K_i \rightarrow  \infty $
    (the sequence $\langle f_i{\rangle}_i$ becomes “harder” to kill with $i$ increasing).

    Let $f_i'\in \Zsuf_1(X,\ZZ)$ be an extension of $f_i$, as in the definition, but with finite support (their existence is proven in \Cref{finite_extension}).
    Since $X$ is infinite and transitive, one can assume that the support of all $f_i'$s are pairwise disjoint.

    Let then $F = \sum_i f_i'$.
    $F \in  \Zsuf_1(X,\ZZ)$ implies the existence of some $G \in  \Cuf_2(X,\ZZ)$ with $\partial G = F$.
    In particular, $\partial G|U_i = f_i$, and since $\|G\|\leq K_i$ for large enough $i$, we get a contradiction.
\end{proof}

Note that we do not actually need as much as transitivity, but merely being able to move finite pieces of the graphs sufficiently far, so that, e.g. having a cocompact action by the automorphism group is enough. 

\begin{lemma}\label{finite_extension}
    Assume $X$ is ULF and one-ended.
    Let $U{\subseteq}_f VX$ and $f \in  \Csuf_1(U,\ZZ)$ satisfying $\partial f \subseteq  \partial U$ and extensible to $\hat f \in  \Zsuf_1(X,\ZZ)$ with $\|\hat f\| = \|f\|$.
    Then $\hat f$ can be taken with finite support.

\end{lemma}
\begin{proof}
    That $f$ can be extended to $\hat f \in  \Zsuf_1(X,\ZZ)$ implies the existence of a finite number of vertices $v_1,\ldots,v_n \in  \partial U$, and integers $m_1,\ldots,m_n$ such that $\partial f = \sum m_i v_i$.
    Since $\partial f \subseteq  \partial U$, we must have $\sum m_i = 0$.
    Note that the values at these vertices must be killed. To do this, one can either connect two vertices of opposite sign outside of $U$, or add an infinite ray. Because both the number of $v_i$ as their coefficients are finite, there are at most finitely  many rays. Consider a ball enveloping the finite structure of $\hat{f}$, such that only the rays are escaping. Take any ray coming from a positive vertex and any ray coming from a negative vertex.  Since $X$ is one-ended, both rays can be connected outside of the chosen ball by a path $p$. If the path does not cross any other rays, we can reconnect locally. If it crosses other rays we can see this as sequence of basis vertices starting in one with a positive coefficient, ending in a negative one. Hence there exist two consecutive rays such that their vertices have opposite sign. This ends the proof.
\end{proof}

The above “finite extension” property for dimension $1$ is crucial in our proof that for transitive infinite graphs, vanishing of $\Hsuf_1$ is equivalent to $\Hsuf_1$-expansion.
Were it to also hold for higher dimension, the proof of~\Cref{transitive_graphs_expansion} would obviously generalise.

\subsubsection{$\Huf$}

In the above, we focused on $\Hsuf$, since $\Hsuf$-expansion is defined in terms of the simplicial complex at hand.
In the (coarse) setting of $\Huf$, this is obviously problematic.
Our aim now is essentially to coarsify (pure) $\Hsuf$-expansion.

%
%
%
%
%
%

\begin{definition}[$\Huf_1$-expansion]
    A ULF graph $X$ is said to have (pure) $\Huf_1$-expansion iff its Rips complexes eventually have (pure) $\Hsuf_1$-expansion.
\end{definition}

The following proposition justifies our interpretation of non-vanishing of $\Huf_1(X,\ZZ)$ in terms of the three phenomena of ends, large circuits and expansion:

\begin{proposition}\label{non_exp_for_good_reasons}
    If $X$ has $\ZZ$-small circuits (i.e. $\cC_r(X,\ZZ) = \cC_\infty (X,\ZZ)$ for some $r$) and is one-ended, then $\Huf_1(X,\ZZ)\neq 0$ implies that $X$ does not have pure $\Huf_1$-expansion.
\end{proposition}
In other words, if $\Huf_1(X,\ZZ)\neq 0$ then at least one of the three phenomena is responsible for it.
In the non-transitive case, it is not clear that the converse holds, that is, that $\Huf_1(X,\ZZ)=0$ would imply at least one of: ends, large circuits, expansion.

\begin{proof}
    \newcommand\hf {\hat f}
    Assume $\Huf_1(X,\ZZ)\neq 0$, $X$ is one-ended and for some $r$, $\cC_r(X,\ZZ) = \cC_\infty (X,\ZZ)$.
    Then, by~\Cref{hsuf_rewritten}, for $s$ big enough, $\Hsuf_1(R_sX,\ZZ)\neq 0$.
    By~\Cref{expansion_implies_Huf_vanishes}, it follows that $R_sX$ does not have expansion, and it remains to check that it dos not have \emph{pure} expansion either.
    
    Fix $(U,f)$ with $f \in  \Csuf_1(R_sU,\ZZ)$ with $\partial f\subseteq  \partial U$ and $f$ extensible.
    
    Since $R_sX$ does not have expansion, it suffices to check that there exists some $g\in \Csuf_2(R_sX,\ZZ)$ with $\partial g|{R_sU} = f$ and $\|\partial g\| = \|f\|$.
    Let $\hf$ a finite extension of $f$, as given by~\Cref{finite_extension}.
    By “tracing virtual edges”, $\hf$ is, up to a finite sum of boundaries of triangles in $R_sX$, an element of $\Zsuf_1(X,\ZZ)$, of finite support. 
    In particular, up to a finite sum of boundaries of triangles, $\hf$ lies in $\cC_\infty (X,\ZZ) = \cC_r(X,\ZZ)$, and can be decomposed into a sum of circuits of length $\leq r$.
    Now, if $s$ is large enough, each of those circuits can itself be decomposed as a sum of boundaries of triangles in $R_sX$, so that $\hf$ can be written as a sum of boundaries of triangles of $R_sX$.
    In other terms, $\hf=\partial g$ for some $g\in \Csuf_2(R_sX,\ZZ)$.

\end{proof}

\begin{theorem}[$\Huf_1$ triad]
    Let $X$ a transitive graph, and consider the following three conditions:
   \begin{enumerate}
       \item \label{ends}
            $X$ has more than one end.
        \item \label{large_circuits}
            $X$ has $\ZZ$-large circuits; that is, $\cC_r(X,\ZZ) \neq  \cC_\infty (X,\ZZ)$ for any $r$.
        \item \label{no_expansion}
            $X$ does not have pure $\Huf_1$-expansion.
    \end{enumerate}
    Then $\Huf_1(X,\ZZ)\neq 0$ iff either of the above holds.
\end{theorem}
\begin{proof}
    
    By the results in~\Cref{section:trees}, know already that~\Cref{ends} implies non-zero homology.
    Similarly, the results in~\Cref{section:ZZ_large_circuits} give that~\Cref{large_circuits} implies non-zero homology.
    
    Finally, assume that~\Cref{no_expansion} holds.
    If any of~\Cref{ends,large_circuits} hold, then obviously $\Huf_1(X,\ZZ)\neq 0$ by what we already know.
    So, one may assume now that for some $r$ large enough, $\cC_r(X,\ZZ)=\cC_\infty (X,\ZZ)$.
    Since $R_rX$ does not have (pure) expansion ($r$ large enough), $\Hsuf_1(R_rX,\ZZ)\neq 0$, and by~\Cref{hsuf_rewritten}, $\Huf_1(X,\ZZ)\neq 0$ either.

    For the converse, it suffices to see that if $\Huf_1(X,\ZZ)\neq 0$, but $X$ has one end and $\ZZ$-small circuits (i.e. neither of \Cref{ends,large_circuits} hold), then we can apply \Cref{non_exp_for_good_reasons}

\end{proof}